\documentclass[preprint]{imsart}
\RequirePackage[OT1]{fontenc}
\RequirePackage{amsthm,amsmath,graphicx}
\RequirePackage[sort,numbers]{natbib}
\usepackage{verbatim}
\RequirePackage[colorlinks,citecolor=red,urlcolor=blue]{hyperref}

\usepackage{amsfonts,amssymb}
\usepackage[ruled,vlined]{algorithm2e}
\usepackage{subcaption}
\usepackage{float}

\startlocaldefs

\theoremstyle{plain}
\newtheorem{lemma}{Lemma}[section]

\newtheorem{proposition}{Proposition}[section]
\newtheorem{theorem}{Theorem}[section]
\newtheorem{definition}{Definition}[section]
\newtheorem{corollary}{Corollary}[section]

\DeclareMathOperator*{\argmin}{arg\,min}
\DeclareMathOperator*{\argmax}{arg\,max}

\let\oldenumerate\enumerate
\renewcommand{\enumerate}{
	\oldenumerate
	\setlength{\itemsep}{1pt}
	\setlength{\parskip}{0pt}
	\setlength{\parsep}{0pt}
}

\newcommand{\sgn}[1]{\text{sgn} \left ( #1 \right )}

\endlocaldefs

\begin{document}

\begin{frontmatter}

\title{AdaBoost and robust one-bit compressed sensing}
\runtitle{AdaBoost and 1-bit CS}

\author{\fnms{Geoffrey} \snm{Chinot}\ead[label=e1]{geoffrey.chinot@stat.math.ethz.ch}},
\author{\fnms{Felix} \snm{Kuchelmeister}\ead[label=e2]{felix.kuchelmeister@stat.math.ethz.ch}},
\author{\fnms{Matthias} \snm{L\"offler}\ead[label=e3]{matthias.loeffler@stat.math.ethz.ch}}
\and 
\author{\fnms{Sara} \snm{van de Geer}\ead[label=e4]{sara.vandegeer@stat.math.ethz.ch}}
\affiliation{ETH Z\"urich \thanksmark{a}}
\address{Seminar for Statistics, Department of Mathematics, ETH Z\"urich, Switzerland,\\
Emails: \printead*{e1,e2}; \\ \printead*{e3,e4} }

\begin{abstract}
This paper studies binary classification in robust 
one-bit compressed sensing with adversarial errors. It is assumed that the model is overparameterized and that the parameter of interest is effectively sparse. AdaBoost is considered, and, through its relation to the max-$\ell_1$-margin-classifier, prediction error bounds are derived. The developed theory is general and allows for heavy-tailed feature distributions, requiring only a weak moment assumption and an anti-concentration condition. Improved convergence rates are shown when the features satisfy a small deviation lower bound. 
In particular, the results provide an explanation why interpolating adversarial noise can be harmless for classification problems. 
Simulations illustrate the presented theory. 
\end{abstract}




\end{frontmatter}

\section{Introduction}
Classification is a fundamental statistical problem in data science, with applications ranging from genomics to character recognition. AdaBoost, proposed by Freund and Schapire \cite{FreundSchapire97} and further developed in \cite{SchapireSinger99}, is a popular and successful algorithm from the machine learning literature to tackle such classifications problems.  It is based on building an additive model with coefficients $\tilde \beta_T$ composed of simple classifiers such as regression trees and then using the binary classification rule $\textnormal{sgn}(\langle \tilde \beta_T, \cdot \rangle)$. At each iteration another  simple classifier is added to the model, minimizing a weighted loss-function. Alternatively, AdaBoost can be viewed as a variant of mirror-gradient-descent for the exponential loss \cite{Breiman98,FriedmanHastieTibshirani00}.  Empirically, it often achieves the best generalization performance when it is overparameterized and runs long after the  training error equals zero \cite{DruckerCortes96}.

However, a theoretical understanding of the generalization properties of AdaBoost, that explains this behaviour, is still missing. Early theoretical results on the generalization error of AdaBoost and other classification algorithms were based on margin-theory \cite{BartlettFreundLeeSchapire98,KoltchinskiiPanchenko02} and entropy bounds. In high-dimensional situations, where the dimension of the features and number of base classifiers is larger than the number of observations $n$, these become meaningless.  Another approach to explain the success of AdaBoost and other boosting algorithms is based on regularization through early stopping \cite{Jiang04, ZhangYu05,Buehlmann06}. However, by their nature these bounds can not explain generalization performance when the number of iterations grows large and the empirical training error equals zero. In the population setting \cite{Breiman04} showed that the generalization  risk of AdaBoost converges to the Bayes risk, but this does also not indicate any performance guarantees for finite data. 

A more thorough understanding has developed through the lens of optimisation. Already in \cite{FreundSchapire97} it was shown that each iteration of AdaBoost decreases the training error. Moreover, in \cite{Breiman98,FriedmanHastieTibshirani00}, a close connection to the exponential loss was pointed out and studied. Building on these results, \cite{RaetschOnodaMueller01,ZhangYu05,RossetZhuHastie04} discovered that overparameterized AdaBoost, when run long enough with vanishing learning rate $\epsilon$ (see Algorithm~\ref{alg:main}), has $\ell_1$-margin converging to the maximal $\ell_1$-margin. In particular, this means that given training data $(X_i,y_i)_{i=1}^n$, where the $y_i$ are binary and the $X_i$ are  $p$-dimensional feature vectors, and where $\tilde \beta_T$ denotes the output of AdaBoost with the canonical basis as simple classifiers, learning rate $\epsilon$ and run-time $T$, we have 
\begin{align} \label{def max margin}
  \min_{1 \leq i \leq n} \frac{ \langle y_iX_i, \tilde \beta_T\rangle }{\|\tilde \beta_T\|_1} \xrightarrow{\substack{T \rightarrow \infty\\ \epsilon \rightarrow 0}} \max_{\beta \neq 0} \min_{1 \leq i \leq n} \frac{\langle y_i  X_i, \beta \rangle}{\|\beta\|_1}=:\gamma,
\end{align}
provided that $\gamma$ is positive. The above holds universally for boosting algorithms that are derived from exponential type loss functions and various possible adaptive step-sizes. For these, general non-asymptotic bounds have been developed in  \cite{MukherjeeRudinSchapire13,Telgarsky13}. 

Any vector $\hat \beta$ that maximizes the right hand side in \eqref{def max margin} is 
proportional to an output of
\begin{align} \label{intro def estimator min l1}
   \hat \beta \in \argmin \left \{\left \| \beta \right \|_1 ~~\text{subject to}~~\min_{1 \leq i \leq n} y_i \langle X_i, \beta \rangle \geq 1 \right \}.
\end{align}
From the representation \eqref{intro def estimator min l1}, it can be seen that, if $\hat \beta$ is well-defined, then $\hat \beta$ interpolates the data  in the sense that $\langle X_i, \hat \beta \rangle$ and $y_i$ have matching signs for all $i$. Similarly, neural networks and random forests are typically massively overparameterized and trained until they interpolate the data.  Empirically, it has been shown that this can lead to smaller test errors compared to algorithms with a smaller number of parameters \cite{WynerOlsonBleichMease17,BelkinHsuMaMandal19}. Statistical learning theory based on empirical risk minimization techniques and entropy bounds can not explain these empirical findings and a mathematical understanding of this phenomenon has only began to form in recent years. The prevalent explanation so far is that, similar as in \eqref{def max margin}, these algorithms approximate max-margin solutions \cite{Telgarsky13,SoudryHofferNacsonGunasekarSrebro18,JiTelgarsky19}. As in  \eqref{intro def estimator min l1}, an algorithm that maximises a margin is equivalent to a minimum-norm-interpolator. It is then argued that this leads to implicit regularization and hence a good fit. 

The study of minimum-norm interpolating algorithms has mainly been investigated in three settings so far. The first line of research has focused on a random matrix regime where the number of data points and parameters are proportional. Here precise asymptotic results can be obtained, see for instance \cite{MontanariRuanSohnYan20,DengKammounThrampoulidis20} for max-$\ell_2$-margin interpolation, \cite{LiangSur20} for max-$\ell_1$-margin interpolation and consequently AdaBoost, \cite{MeiMontanari19} for 2-layer-neural networks in regression and \cite{HastieMontanariRossetTibshirani19} for minimum-$\ell_2$-norm linear regression. However, these results do not exploit possible  low-dimensional structure such as sparsity and they also require a large enough, constant, noise-level, leading to inconsistent estimators. 

Another line of work has focused on non-asymptotic results in an Euclidean setting with features that have a covariance matrix  with decaying eigenvalues, see \cite{MuthukumarNarangSubramanianBelkinHsuSahai20} for classification with support-vector machines (SVM) and \cite{BartlettLongLugosiTsigler20,ChinotLerasle20} for linear regression.  These results rely crucially on Euclidean geometry, which gives explicit formulas for the estimators under consideration, and also do not lead to improved convergence rates in the presence of low-dimensional intrinsic structure.

A third line of work originates in the compressed sensing literature. Here low-dimensional intrinsic structure and often small noise levels, including adversarial noise,  are studied. 
Small noise might be a realistic assumption for many classification data sets from the machine learning literature. On  data sets such as CIFAR-10 \cite{ForetNeyshaburKleinerMobahi21} or MNIST \cite{WanZeilerZhangLecunFergus13} state of the art algorithms achieve  test errors smaller than $0.5\%$, implying that the proportion of flipped labels in the full data set is also small.
On the theoretical side, pioneering work by Wojtaszczyk \cite{Wojtaszczyk10} has shown that minimum-$\ell_1$-norm interpolation, introduced by \cite{ChenDonohoSaunders98} as {\em basis pursuit},  is robust to small, adversarial errors in sparse linear regression. This has recently been extended to other minimum-norm-solutions in linear regression \cite{ChinotLofflervandeGeer20}, phase-retrieval \cite{KrahmerKuemmerleMelnyk20} and heavy-tailed features in sparse linear regression \cite{KrahmerKuemmerleRauhut18}. 

Sparsity enables to model the possibility that only few variables are sufficient to predict well and allows for easier model interpretation. 
In binary classification, a sparse model with adversarial errors can be described by having access to a dataset $ \mathcal D_n = (X_i,y_i)_{i=1}^n$, where the features $(X_i)$'s are i.i.d random vectors in $\mathbb{R}^p$ distributed as $X$ and $X = (x_1,\cdots,x_p)$ where $x_j \overset{i.i.d.}{\thicksim} \mu$ for some distribution $\mu$. For $s>0$ we are given an effectively $s$-sparse $\beta^* \in \mathcal{S}^{p-1}$, i.e. a vector $\beta^*$ such that $\|\beta^*\|_2=1$ and $\| \beta^* \|_1 \leq \sqrt s$. Finally, for a set $\mathcal{O} \subset [n]$ we have
\begin{align}
    y_i = \begin{cases} \sgn{\langle X_i, \beta^*\rangle} ~~~~~~~i \notin \mathcal{O} \\
    -\sgn{\langle X_i, \beta^*\rangle}~~~~~~i \in \mathcal{O}. \end{cases} \label{intro model}
\end{align}
The set $\mathcal{O}$ contains the indices of the data that is labeled incorrectly.  We do not impose any modelling assumptions on $\mathcal{O}$. $\mathcal{O}$ may be random, deterministic or adversarially depend on all features $(X_i)_{i=1}^n$, but we impose that the proportion of flipped labels is small such that $|\mathcal{O}|=o(n)$. 
In the applied mathematics literature, this model is called {\em robust one-bit compressed sensing} and in learning theory {\em agnostic learning of (sparse) half-spaces}. 

As far as we know, there are no theoretical  results for estimators that necessarily  interpolate in the model \eqref{intro model} when $\mathcal{O} \neq \emptyset$. In the noiseless case where $\mathcal{O}=\emptyset$ and for standard Gaussian measurements,
 \cite{PlanVershynin13CPAM} have proposed and investigated an interpolating estimator, similarly defined as \eqref{intro def estimator min l1} with the minimum replaced by an average and an additional matching sign constraint. In particular, they showed that this estimator is able to consistently estimate the direction of $\beta^*$. 
 
 Subsequent work where the model \eqref{intro model} and variants of it were considered, has  focused on regularized estimators in order to adapt to noise or to generalize the required assumptions. First results for the model \eqref{intro model} and a computable algorithm were obtained by \cite{PlanVershynin13}, where a convex program was proposed and investigated. If $\beta^*$ is exactly $s$-sparse, i.e. it has at most $s$ non-zero entries, the attainable convergence rates can be improved and faster performance guarantees were obtained by \cite{JacquesLaskaBoufounosBaraniuk13,ZhangYiJin14,AwasthiBalcanHaghtalabZhang16}. Further works investigated non-Gaussian measurements \cite{AiLapanowskiPlanVershynin14}, active learning \cite{AwasthiBalcanHaghtalabZhang16,Zhang18,ZhangShenAwasthi20}, overcomplete dictionaries \cite{BaraniukFoucartNeedell18} and random shifts of $\langle X_i, \beta^*\rangle$, called {\em dithering}, \cite{KnudsonSaabWard16,DirksenMendelson21}. 
 
In this paper, we consider the performance of AdaBoost in the overparameterized regime with small and adversarial noise. We additionally assume that $\beta^*$ is effectively $s$-sparse. We leverage the relation in \eqref{def max margin} between AdaBoost and the max-$\ell_1$-margin estimator \eqref{intro def estimator min l1} to analyze AdaBoost (as described below in Algorithm \ref{alg:main}). In particular, we show that when $p \gtrsim n$ and the feature vectors fulfill a weak moment assumption and an anti-concentration assumption, then with high probability AdaBoost has vanishing prediction error, provided $\left(s+ |\mathcal{O}| \right)\log^c(p)=o(n)$ for some constant $c>0$
and sufficiently many iterations $T=O(n)$ of AdaBoost are performed. Moreover, when the features are Gaussian or student-t (with at least $c\log(p)$ degrees of freedom) distributed, we obtain prediction and Euclidean estimation error bounds that scale as
\begin{align} \label{intro rate 1/3}
  \left (  \frac{\left(s+|\mathcal{O}| \right) \log^c(p)}{n} \right )^{1/3},
\end{align}
which is among the best available convergence guarantees in the one-bit compressed sensing literature so far.

These results are, as far as we know, the first non-asymptotic guarantee for  overparameterized and data interpolating AdaBoost in a sparse and noisy setting. 
We illustrate our theory with Laplace, uniform, Gaussian and student-t (with at least $c\log(p)$ degrees of freedom) distributed features. 
Moreover, our main result also explains why interpolating data can perform well in the presence of adversarial noise, providing an explanation to the question raised in~\cite{ZhangBengioHardtRechtVinyals17}. Numerical experiments complement our theoretical results. \\
Compared to \cite{LiangSur20} we consider a completely different regime. In their setting sparsity can not be assumed and the noise level can neither be adversarial nor small. Hence, in \cite{LiangSur20} consistent estimation of the direction of $\beta^*$ is impossible and the resulting generalization error is close to $1/2$ when $p$ is large compared to $n$.

\subsection*{Notation}
The Euclidean norm is denoted by  $\|\cdot \|_2$ and induced by the inner product $\langle \cdot , \cdot \rangle$, $\| \cdot \|_1$ denotes the $\ell_1$-norm  and $\| \cdot \|_\infty$ the $\ell_\infty$-norm for both vectors and matrices. $B_1^p$ and $B_2^p$ denote the unit $\ell_1$-ball and $\ell_2$-ball in $\mathbb{R}^p$, respectively. In addition, we write $\mathcal{S}^{p-1}$ for the $p$-dimensional unit sphere. By $c$, we  denote a generic, strictly positive constant, that may change value from line to line. Moreover, for two sequences $a_{n}, b_n$ we write $a_n \lesssim b_n$ if  $a_n \leq c b_n~\forall n$. Similarly, $a_n \gtrsim b_n$ if $b_n \lesssim a_n$ and $a_n \asymp b_n$ if $a_n \lesssim b_n$ and $b_n \lesssim a_n$. When an assumption reads 'Suppose $a_n \lesssim b_n$' this means that we assume that for a small enough constant $c>0$ we have $a_n \leq c b_n~\forall n$. By $[p]$ we denote the enumeration $\{1, \dots, p\}$, by $\{e_j\}_{j \in [p]}$ the set of canonical basis vectors in $\mathbb{R}^p$ and by $X_i$ the $i$-th column of the matrix $\mathbb X=[X_1, \dots, X_n]$ of feature vectors. We denote the sign function, $\sgn{x}=\mathbf{1}(x >0)-\mathbf{1}(x<0)$.
Throughout this article, we use bold letters to denote random matrices, upper case letters for random vectors and lower case letters for random variables. For example we  write $\mathbb X = (X_i)_{i \in [n]} \in \mathbb R^{p \times n}$ and $X_i = (x_{i,1},\cdots,x_{i,p}) \in \mathbb R^p$. 

\section{Main results}
\subsection{Model and assumptions}
We consider a binary classification model, which allows for adversarial flips. In particular, we assume that we have access to a dataset $ ( y_i, X_i)_{i \in [n]}$. The $X_i$'s are i.i.d random vectors in $\mathbb{R}^p$ distributed as $X$ and $X= (x_1,\cdots,x_p)$, where  $x_j \overset{i.i.d.}{\thicksim} \mu$ for some distribution $\mu$ that is symmetric and has zero mean and unit variance. Assuming unit variance is not restrictive, as both the considered loss functions and the observed data are scaling invariant. The $y_i$'s are generated via 
\begin{align}
    y_i = \begin{cases} &\sgn{\langle X_i, \beta^*\rangle} ~~~~~~~i \notin \mathcal{O} \\
    &-\sgn{\langle X_i, \beta^* \rangle} ~~~~~i \in \mathcal{O}.\end{cases} \label{main def model}
\end{align}
The set $\mathcal{O} \subset [n]$ is the set of the indices of the mislabeled data. 
We assume that the fraction of flipped labels is asymptotically vanishing, $|\mathcal{O}|=o(n)$, 
but that $\mathcal{O}$ may be picked by an adversary and depend on the data. In particular, this includes parametric noise models such as logistic regression or additive Gaussian noise inside the sign-function above, as long as the variance of the noise decays to zero as $n$ goes to infinity. 
We assume that $\beta^* \in \mathcal{S}^{p-1}$ is effectively $s$-sparse, that is $\|\beta^*\|_1 \leq \sqrt{s}$.  \\
For stating our results we will treat all other parameters that do not depend on $p, n, s$ and $|\mathcal{O}|$ as fixed constants. Moreover, we always assume tacitly that $p \geq c n$, for a large enough constant $c>0.$

 We measure the accuracy of recovery by the prediction error 
\begin{align} \label{def pred error}
  d(\tilde \beta,\beta^*):=  \mathbb{P} \left ( \sgn{\langle X_{n+1}, \tilde \beta\rangle } \neq \sgn{\langle X_{n+1}, \beta^* \rangle}\bigg |(y_i,X_i)_{i \in [n]} \right ),
\end{align}
where $X_{n+1}$ is an independent copy of $X$. This is the quantity which is used empirically to measure the quality of classifiers such as neural networks on standard image benchmark data sets \cite{WanZeilerZhangLecunFergus13,ForetNeyshaburKleinerMobahi21}. 

We now formulate the three main assumptions used throughout this article. They describe the tail-behaviour and behaviour around zero of the features. 

For the tail-behaviour, the only assumption we make is a weak moment assumption of order $\log(p)$.
    \begin{definition} 
     A centered, scalar random variable $x$ fulfills the weak moment assumption (of order $log(p)$) with parameter 
$\zeta \geq 1/2$ if 
     $$
    \left( \mathbb E | x |^{q}  \right)^{1/q} \lesssim 
   q^{\zeta} \quad \forall ~ 1 \leq q \leq \max(1, \log(p)) .
    $$
For a matrix $\mathbb{X}$ or a vector $X$ we say that they satisfy the weak moment assumption if their entries satisfy the weak moment assumption. 
    \end{definition}
This assumption is weaker than commonly assumed sub-Gaussian or sub-exponential tail behaviour and allows for feature distributions with heavy-tails such as the student-t-distribution with  $c \log(p)$ degrees of freedom. 
Under the weak moment assumption we are able to control the $\ell_{\infty}$ norm of $X = (x_1,\cdots,x_p)$ composed of i.i.d random variables satisfying the weak moment assumption with polynomial deviation (see Proposition~\ref{control infty norm log moments}). Assuming sub-Gaussianity of the $x_j$'s does not lead to improvement of the convergence rates in our main results except for logarithmic factors. This is different to the theory developed in \cite{DirksenMendelson21} where the exponent in the obtained convergence results depends on whether the features are sub-Gaussian or heavy-tailed.  \\
The next two assumptions measure the behaviour of the feature distribution around zero. 

\begin{definition} \label{def anticoncentration} A random vector $X \in \mathbb{R}^p$ fulfills an anti-concentration assumption with parameter 
$\alpha \in (0,1]$ if
\begin{align} \label{def eq anticon}
    \sup_{\beta \in \mathcal{S}^{p-1}} \mathbb{P} (|\langle X, \beta \rangle |\leq \varepsilon) \lesssim  \varepsilon^{\alpha} ~\forall~ p^{-1} \leq \varepsilon \leq 1. 
\end{align}
We say that a matrix $\mathbb{X} \in \mathbb{R}^{p \times n}$ satisfies the anti-concentration assumption with $\alpha \in (0,1]$ if each column of $\mathbb{X}$ satisfies \eqref{def eq anticon}.
\end{definition}
Assuming that $\mathbb X$ satisfies an anti-concentration assumption will be a necessity for our results, as it ensures that for fixed vector $\beta$, the scalar $|\langle X_i, \beta\rangle|$ is not too close to zero for too many indices $i \in [n]$. This in turn would lead to a tiny $\ell_1$-margin and many discontinuities of $\sgn{\langle X_i, \beta \rangle}$ at $\beta$, rendering it impossible to prove uniform results.

Similar anti-concentration assumptions were previously introduced in the learning theory literature in the non-sparse setting, see e.g. \cite{BalcanZhang17,DiakonikolasTzamosKontonisZarifis,FreiCaoGu21}, and were shown to be satisfied by isotropic log-concave distributions \cite{BalcanZhang17} via an uniform upper bound for the density of $\langle \beta,X \rangle$.

The next assumption is an optional counterpart to Definition \ref{def anticoncentration} and leads to improved convergence rates if it is satisfied. 
\begin{definition} \label{def small deviation} A random vector $X \in \mathbb{R}^p$ fulfills a small deviation assumption with parameter $\theta >0$ if
\begin{align} \label{def eq small dev}
    \inf_{\beta \in \mathcal{S}^{p-1}} \mathbb{P} (|\langle X, \beta \rangle |\leq \varepsilon) \gtrsim 
    \varepsilon^{\theta}  ~\forall~ p^{-1} \leq \varepsilon\leq 1. 
\end{align}
We say that a matrix $\mathbb{X} \in \mathbb{R}^{p \times n}$ satisfies the small deviation assumption with parameter $\theta >0$ if each column of $\mathbb{X}$ satisfies \eqref{def eq small dev}.
\end{definition}
In \cite{DiakonikolasTzamosKontonisZarifis} a stronger small-deviation assumption was formulated, assuming an uniform lower bound on the density of two-dimensional projections of $X$. This property is satisfied by isotropic log-concave distributions \cite{BalcanZhang17} and implies \eqref{def eq small dev} with $\theta=1$.

\subsection{Main results}
\subsubsection{AdaBoost, max $\ell_1$-margin and a bound in terms of the margin}

AdaBoost, proposed by Freund and Schapire \cite{FreundSchapire97}, is an algorithm where an additive model for an unnormalized version of $\beta^*$ is built by iteratively adding weak classifiers to the model. To facilitate our  analysis, we assume that the features $X_i$'s are i.i.d. distributed and that the weak classifiers can be identified with the standard basis vectors in $\mathbb{R}^p$. We consider AdaBoost as described in Algorithm  \ref{alg:main}. The main difference to the original proposal by \cite{FreundSchapire97} consists of the choice of the step-size $\alpha_t$, which is obtained by minimizing a quadratic upper bound for the loss-function at each step \cite{Telgarsky13}. 
\begin{algorithm}[h]
\SetAlgoLined
\KwIn{Binary data  $(y_{i})_{i \in [n]}$,  features $\mathbb{X}=(X_i)_{i \in [n]}$, run-time $T$, learning rate $\epsilon$ }
\KwOut{Vector $\tilde \beta_T \in \mathbb{R}^p$}
\nl Initialize $\tilde \beta_{0,i}=0$ and rescale features $\mathbb{X}=\mathbb{X}/\|\mathbb{X}\|_\infty$\\
 \nl For $t=1, \dots T$ repeat
 \begin{itemize}
     \item Update weights $w_{t,i}
 =\frac{ \exp(-y_i \langle X_i, \tilde \beta_{t-1} \rangle )}{\sum_{j=1}^n \exp(-y_j \langle X_j, \tilde \beta_{t-1} \rangle )} $, $i=1,\dots,n$
     \item  Select coordinate:~$v_{t}=\argmax_{v \in \{ e_j \}_{j=1}^p } | \sum_{i=1}^n  w_{t,i} y_i \langle X_i, v \rangle |$
     \item Compute adaptive stepsize $\alpha_t = \sum_{i=1}^n w_{t,i} y_i \langle X_i, v_{t} \rangle $
     \item Update $\tilde \beta_{t}=\tilde \beta_{t-1}+\epsilon \alpha_t v_{t}$
    
 \end{itemize}
 \nl Return $\tilde \beta_{T}$ 
\caption{AdaBoost for binary classification}\label{alg:main}
\end{algorithm}

Alternative to the interpretation by Freund and Schapire \cite{FreundSchapire97}, AdaBoost can be viewed as a form of mirror gradient descent on the exponential loss-function \cite{Breiman98,FriedmanHastieTibshirani00}. It  is thus natural to expect that it converges to the infimum of the loss-function and eventually interpolates the labels if possible. 
In fact, a stronger statement holds: As described in \eqref{def max margin}, 
AdaBoost with infinitesimally small learning rate and a growing number of iterations $T$ converges to a solution that maximizes the $\ell_1$-margin \cite{RaetschOnodaMueller01,ZhangYu05,Telgarsky13}. 

This holds even non-asymptotically \cite{Telgarsky13} for many variants of AdaBoost and includes both the exponential and logistic loss-function as well as various choices of adaptive stepsizes $\alpha_t$, for instance logarithmic as originally proposed by \cite{FreundSchapire97}, line search \cite{SchapireSinger99,ZhangYu05} or quadratic as in Algorithm \ref{alg:main}.

To present non-asymptotic results and to ensure that our theory can potentially be applied to other variants of AdaBoost, we introduce the following definition of an approximation of the largest $\ell_1$-margin: We say that $\tilde \beta \in \mathbb{R}^p$ provides an approximation of the max $\ell_1$-margin if
\begin{align} \label{def 1-eps max l1 margin}
    \min_{1 \leq i \leq n} \frac{y_i \langle X_i, \tilde \beta \rangle}{\|\tilde \beta\|_1} \geq \frac{1}{2}\max_{\beta \neq 0} \min_{1 \leq i \leq n}\frac{ y_i \langle X_i, \beta \rangle}{ \|\beta\|_1}=:\frac{\gamma}{2}, 
\end{align}
The quantity $\gamma$ is called the max $\ell_1$-margin. Moreover, the factor $1/2$ can be substituted by any other positive constant smaller than one.

The following theorem gives a bound for the prediction error for any $\tilde \beta$ that provides an approximation of the max $\ell_1$-margin. The bound itself depends on the max $\ell_1$-margin $\gamma$. If the features fulfill an additional small deviation assumption, we obtain improved convergence rates. 

\begin{theorem} \label{theorem error bound dependance margin}
Assume $p \gtrsim n$ and that  $\mathbb{X}=(X_i)_{i \in [n]}$ has i.i.d. zero mean and unit variance entries and satisfies the weak moment assumption with $\zeta \geq 1/2$ and the anti-concentration assumption  with 
$\alpha \in (0,1]$. 
Let $\tilde \beta$ be an approximation of the margin and suppose that $\tilde \beta$  satisfies with probability at least $1-t$ that $\gamma \geq \gamma_0$. Define
$$
 \eta =  \left( \frac{\log^{2\zeta+1}(p)\log(n)}{\gamma_0^2 n}   \right)^{\frac{1}{2+\alpha}},
$$
and assume that $\eta \lesssim 1$. Moreover, assume that 
$
| \mathcal O | \lesssim \eta^{\alpha}n.
$
Then with probability at least $1-cp^{-1}-t$ we have that
$$
d \left({\tilde \beta}, \beta^* \right) \lesssim \eta^{\alpha}.
$$
Moreover, if $\mathbb{X}$ satisfies a small deviation assumption with  $\theta >0$ and 
$
| \mathcal O | \lesssim \eta^{\alpha \left ( 1+\frac{2}{\theta}\right)}n,
$
then, with probability at least $1-c p^{-1}-t$, we have that
$$
d \left({\tilde \beta}, \beta^* \right) \lesssim \eta^{\alpha\left (1+\frac{2}{\theta} \right )}. 
$$
\end{theorem}
The proof of Theorem \ref{theorem error bound dependance margin} involves two main arguments: a bound for the ratio $\|\tilde \beta\|_1/\|\tilde \beta\|_2$ in terms of the max $\ell_1$-margin and a sparse hyperplane tesselation result that adapts a proof technique introduced by \cite{DirksenMendelson21}. For the bound on the ratio  we argue by contradiction and show that with high probability no $\beta$ can simultaneously approximate the margin and have small Euclidean norm. If the small deviation assumption is satisfied we obtain an improved bound on the ratio by using a more involved discretisation argument via Maurey's emirical method \cite{Carl85,RudelsonVershynin08,GuedonLecuePajor13}. For the sparse hyperplane tesselation result, we argue similarly as \cite{DirksenMendelson21}, but use again Maurey's empirical method instead of their net argument.  Compared to a discretisation argument via nets (as in \cite{DirksenMendelson21,PlanVershynin13}) this has the advantage that we are able to deal with features that only fulfill the weak moment assumption, while still retaining the same rate (up to logarithmic factors) as in the sub-Gaussian case. By contrast, the obtained convergence rates in \cite{DirksenMendelson21} depend on whether the features are sub-Gaussian or not.

The following lemma shows that AdaBoost, as described in Algorithm \ref{alg:main}, provides an approximation of the max $\ell_1$-margin, when it is run long enough. The proof is a simple adaptation of results by \cite{Telgarsky13} to our setting. 
\begin{lemma}  \label{lemma upper bound iterations}
Consider the AdaBoost Algorithm \ref{alg:main} and suppose that $p \gtrsim n$ and that $\mathbb{X}=(X_i)_{i \in [n]}$ satisfies the weak moment assumption with $\zeta \geq 1/2$. Suppose that $\gamma>0$, that the learning rate $\epsilon$ satisfies $\epsilon \leq 1/6$ and that $$T \gtrsim \log^{2\zeta+1}(np) /(\epsilon^2  \gamma^2).$$  Then, the output of Algorithm \ref{alg:main} provides an  approximation of the max $\ell_1$-margin with probability at least $1-p^{-1}$. 
\end{lemma}
Hence, both for algorithmic (Lemma~\ref{lemma upper bound iterations}) as well as recovery guarantees (Theorem~\ref{theorem error bound dependance margin}), it is necessary to obtain a lower bound on the max $\ell_1$-margin $\gamma$. 

\subsubsection{A bound for the max $\ell_1$-margin}
In this section, we obtain a lower on the max $\ell_1$-margin $\gamma$, holding with large probability. 
\begin{theorem} \label{theorem lower bound margin}
Assume that $p \gtrsim n$ and that $\mathbb{X}=(X_i)_{i \in [n]}$ has i.i.d. symmetric, zero mean and unit variance entries and satisfies the weak moment assumption with $\zeta \geq 1/2$ and the anti-concentration assumption  with 
$\alpha \in (0,1]$.  Then, we have with probability at least $1-cn^{-1}$ that 
\begin{align} \label{eq lower bound margin} 
\gamma \gtrsim \left [ \frac{n}{\log(ep/n)} \left(  s + \frac{\log^{1+2\zeta}(n)| \mathcal O| }{\log(ep/n)} + \log^{1+ 2\zeta}(n)\right)^{\frac{\alpha}{2}} \right ]^{-\frac{1}{2+\alpha}}. 
\end{align}
\end{theorem}
Crucial for the proof of Theorem \ref{theorem lower bound margin} is the fact that, defining
\begin{align} \label{def min l1 estimator}
    \hat \beta \in \argmin \left \{ \|\beta\|_1 ~~\text{subject to}~~y_i \langle X_i, \beta \rangle \geq 1, ~~i=1, \dots, n\right \},
\end{align}
we have the relation $\gamma=1/\|\hat \beta\|_1$ (see Lemma~\ref{lemma margin sur}). Hence, to obtain a lower bound for $\gamma$ it suffices to obtain an upper bound for $\|\hat \beta\|_1$, which we accomplish by explicitly constructing a $\beta$ that fulfills the constraints in \eqref{def min l1 estimator}. In particular, we use the $\ell_1$-quotient property \cite{Wojtaszczyk10,KrahmerKuemmerleRauhut18}
to find a perturbation of $\beta^*$ that has sufficiently small  $\ell_1$-norm while still fulfilling the constraint $y_i \langle X_i, \beta \rangle \geq 1$  for all $ i \in [n]$. 

The following proposition shows that even in an idealized setting with no noise and isotropic Gaussian features, where $\alpha=1$ (see Corollary~\ref{cor rate Gauss studentt}), the lower bound on the margin in Theorem \ref{theorem lower bound margin} is, in general, tight (up to logarithmic factors).  

\begin{proposition} \label{prop margin upper bound}
Suppose $p \gtrsim n$, $\mathcal{O}=\emptyset$ and that the entries of $\mathbb{X}$ are i.i.d. $\mathcal{N}(0,1)$ distributed. Then, for any $\beta^* \in \mathcal{S}^{p-1}$ which satisfies $\|\beta^*\|_\infty \lesssim 1/\sqrt{s}$, we have that
\begin{align} \label{margin lower}
    \mathbb{E} \gamma \lesssim \left ( \frac{\log(p) }{ n} \frac{1}{\sqrt{s}}\right )^{1/3}. 
\end{align}
\end{proposition}

\subsubsection{Rates for AdaBoost}

Combining Theorems \ref{theorem error bound dependance margin} and \ref{theorem lower bound margin} with Lemma~\ref{lemma upper bound iterations}  we obtain the following corollary, that shows convergence rates for AdaBoost. 
\begin{corollary} \label{cor adaboost quadratic stepsize guarantee 1/12}
Grant the assumptions of Theorem  \ref{theorem lower bound margin} and assume that for some large enough constant $\kappa_1=\kappa_1(\alpha,\zeta)$ the AdaBoost Algorithm \ref{alg:main} is run for $$T \gtrsim \left ( n \left(  s + | \mathcal O| \right)^{\frac{\alpha}{2}} \right )^{\frac{2}{2+\alpha}}\log^{\kappa_1}(p)\epsilon^{-2}$$ iterations with learning rate $\epsilon \leq 1/6$. Then, with probability at least $1-cn^{-1}$, the output $\tilde \beta_T$ of AdaBoost Algorithm \ref{alg:main} satisfies for some constant $\kappa_2=\kappa_2(\alpha,\zeta)$
$$
d \left({\tilde \beta_T}, \beta^* \right) \lesssim \left( \frac{(s+|\mathcal{O}|)\log^{\kappa_2}(p)}{n} \right)^{\frac{\alpha}{(2+\alpha)^2}}.
$$
Moreover, if $\mathbb{X}$ satisfies a small deviation assumption with $\theta >0$, then, with probability at least $1-cn^{-1}$
$$
d \left({\tilde \beta_T}, \beta^* \right) \lesssim  \left( \frac{(s+|\mathcal{O}|)\log^{\kappa_2}(p)}{n} \right)^{\frac{\alpha \left (1+\frac{2}{\theta} \right )}{(2+\alpha)^2}}. 
$$
\end{corollary}
As for consistency $(s+|\mathcal{O}|)\log^{\kappa_2}(p)=o(n)$ is required, it is ensured that in this relevant regime AdaBoost is an approximation of the max $\ell_1$-margin if we run AdaBoost for $T \asymp n \log(p)^{\kappa_1}/\epsilon^{-2}$ iterations. Hence, by contrast to other algorithms such as gradient descent (e.g. section 9.3.1 in \cite{boyd2004convex}) where often a logarithmic number of iterations in $n$ suffices, we require in the worst case an approximately linear number of iterations in $n$ to ensure consistency of AdaBoost. 

\subsubsection{Examples} 
We now illustrate our developed theory for some specific feature distributions. First, for the density of the $x_j$'s being continuous, bounded and unimodal, we are able to show that the anti-concentration condition holds with parameter $\alpha=1/2$.
\begin{corollary} \label{cor rate unimodal}
Assume that $\mathbb{X}=(X_i)_{i \in [n]}$ has i.i.d. symmetric, zero mean and unit variance entries and satisfies the weak moment assumption with $\zeta \geq 1/2$. Assume that the $x_{ij}$'s have density $f$ that is continuous, bounded by a constant, and unimodal, i.e. $f(a\varepsilon)\geq f(\varepsilon) ~\forall a \in (0,1), \varepsilon \in \mathbb{R}$. Then $\mathbb{X}$ satisfies the anti-concentration condition with parameter $\alpha=1/2$.  In particular, this includes features that are distributed according to the uniform, Gaussian, student-t with $d \gtrsim \log(p)$, $d \in \mathbb{N},$ degrees of freedom distributions (with $\zeta=1/2$) and the Laplace distribution (with $\zeta=1$). Hence, when $p \gtrsim n$ and AdaBoost is for some constant $\kappa_1=\kappa_1(\zeta)$ run for at least 
$$T \gtrsim \left ( n \left(  s + | \mathcal O| \right)^{\frac{1}{4}} \right )^{\frac{4}{5}}\log(p)^{\kappa_1}\epsilon^{-2}$$
iterations, then with probability at least $1-cn^{-1}$ we have that for some constant $\kappa_2=\kappa_2(\zeta)$
\begin{align*}
    d(\tilde \beta_T, \beta^*) \lesssim \left ( \frac{(s+|\mathcal{O}|)\log^{\kappa_2}(p)}{n}\right )^{\frac{2}{25}}. 
\end{align*}
\end{corollary}
When the features are Gaussian or student-t with at least $c\log(p) $ degrees of freedom distributed, we are able to improve upon this and show that the anti-concentration and small deviation conditions are both fulfilled with parameters $\alpha=\theta=1$. Moreover, for these distributions the prediction and Euclidean estimation error are closely related such that we are also able to obtain error bounds in this distance. 
\begin{corollary} \label{cor rate Gauss studentt}
 Assume that the entries of $\mathbb{X}=(X_i)_{i \in [n]}$ are i.i.d. $\mathcal{N}(0,1)$ or $\sqrt{(d-2)/d} t_{d}$ distributed for $ \log(p) \lesssim d $, $d \in \mathbb{N}$, $p \gtrsim 1$. Then $\mathbb{X}$ satisfies the anti-concentration and small deviation assumptions with $\alpha=\theta=1$ and the weak moment assumption with $\zeta=1/2$. In particular, when $p \gtrsim n$, and after at least 
$$T \gtrsim \left ( n \left(  s + | \mathcal O| \right)^{\frac{1}{2}} \right )^{\frac{2}{3}}\log(p)^{\kappa_1}\epsilon^{-2}$$
 iterations of AdaBoost, we have with probability at least $1-cn^{-1}$ that for some constant $\kappa_2$
 \begin{align*}
       d(\tilde \beta_T, \beta^*) \lesssim \left (  \frac{(s+|\mathcal{O}|)\log^{\kappa_2}(p)}{n} \right )^{1/3}. 
 \end{align*}
 Moreover, on the same event, we have that 
 \begin{align} \label{cor Euclidean rates Ada}
 \left  \| \frac{\tilde \beta_T}{\|\tilde \beta_T\|_2} - \beta^*\right \|_2 \lesssim \left (  \frac{(s+|\mathcal{O}|)\log^{\kappa_2}(p)}{n} \right )^{1/3}. 
 \end{align}
\end{corollary}
   We now compare the convergence guarantees for AdaBoost with Gaussian or student-t distributed features with the state of the literature, where mostly Gaussian features and Euclidean estimation error were considered. 
    When $\mathcal{O}=\emptyset$ the performance guarantee in \eqref{cor Euclidean rates Ada} is better than existing bounds for  regularized algorithms \cite{PlanVershynin13,ZhangYiJin14} and match, up to logarithmic factors, the best available bounds that can be obtained by combining the tesselation result in Proposition \ref{proposition tesselation maurey} with Plan and Vershynin's \cite{PlanVershynin13CPAM} linear programming estimator. A straightforward adaptation of the proofs from \cite{DirksenMendelson21} for the tesselation to our setting, would lead to an exponent of $1/4$ in \eqref{cor Euclidean rates Ada} in case of the student-t distribution with at least $c \log(p)$ degrees of freedom. We achieve improved rates by replacing the net dicretization from~\cite{DirksenMendelson21} with a more involved Maurey  argument. 
    
For Gaussian features and in the presence of adversarial errors, the convergence rate obtained in \eqref{cor Euclidean rates Ada} improves over the rate for the regularized estimator by \cite{PlanVershynin13} if  $(|\mathcal{O}|/n)^{4}=o(s/n)$ and otherwise their algorithm achieves faster convergence rates, in both cases up to logarithmic factors.  If $\beta^*$ is exactly $s$-sparse, i.e. at most $s$ entries of $\beta^*$ are non-zero, then the rate in \eqref{cor Euclidean rates Ada} is sub-optimal in the dependence on $s\log(p)/n$ and $|\mathcal{O}|/n$ and faster rates were obtained for a (non-interpolating) regularized estimator in \cite{AwasthiBalcanHaghtalabZhang16} for strongly log-concave features. \\

\subsection{Simulations}

In this subsection, we provide simulations for various feature distributions to illustrate our theoretical results qualitatively.
Alongside Theorem \ref{theorem error bound dependance margin}, we show the empirical prediction error as a function of the sample size $n$ and the number of corrupted labels $|\mathcal{O}|$.
Moreover, to accompany Theorem \ref{theorem lower bound margin}, we plot the margin as a function of $n$.

As illustrated in Corollary \ref{cor rate unimodal}, the developed theory applies to various distributions of the entries of the features $X_{ij}$, such as continuous, bounded and unimodal distributions.
To highlight the universality of our theory, simulations were performed for the standard normal distribution $\mathcal{N}(0,1)$, the student-t distribution with $\log(p)$ degrees of freedom, the uniform distribution with unit variance , and the Laplace distribution with zero location and unit scale parameter. 

The ground truth $\beta^*$ was generated randomly, with an $s$-sparse Rademacher prior.
That is, $s$ out of the possible $p$ entries are chosen at random, and set to $\pm1/\sqrt{s}$ with equal probability.
The remaining entries are set to zero, making $\beta^*$ $s$-sparse, with $||\beta^*||_2=1$.
The indices for the set of corruptions $\mathcal{O}$ was chosen uniformly at random , such that a predetermined number of labels is corrupted.
The sparsity was chosen as $s=5$.
As Theorem \ref{theorem error bound dependance margin} assumes $p\gtrsim n$, we let $p=10n$.
AdaBoost was executed as described in Algorithm \ref{alg:main}, using step size $\epsilon=0.2$. 
The number of steps performed was $T=(n\sqrt{s+|\mathcal{O}|})^{2/3}\log(p)/\epsilon^2$ steps, imitating the setting in Corollary \ref{cor rate Gauss studentt}.
The simulated are averaged over twenty iterations.

\begin{figure}[H]

    \centering
    \begin{subfigure}{0.5\textwidth}
        \includegraphics[width = \linewidth]{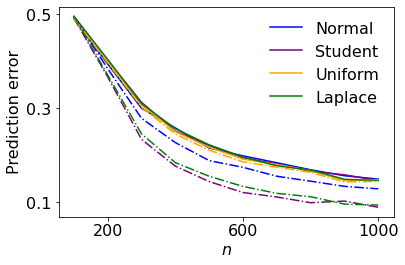}
    \end{subfigure}%
    \begin{subfigure}{0.5\textwidth}
        \includegraphics[width = \linewidth]{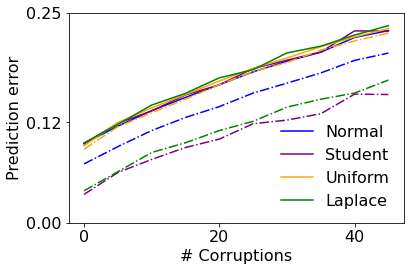}
    \end{subfigure}

    \caption{On the left, we plot the prediction error for  $|\mathcal{O}|=40$ corruptions, against the number of samples $n$, for various features. On the right, we show for $n=500$ how the prediction error changes as the number of randomly flipped labels $|\mathcal{O}|$ decreases. The solid lines represent the max-$\ell_1$-margin estimators $\hat{\beta}$ \eqref{intro def estimator min l1}. The dash-dotted lines are instances of AdaBoost $\tilde{\beta}_T$, as defined in Algorithm \ref{alg:main}.}
       \label{figure 1}

    \begin{subfigure}{0.5\textwidth}
        \includegraphics[width = \linewidth]{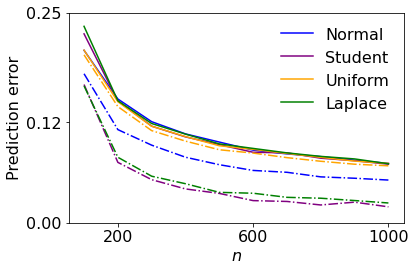}
   
    \end{subfigure}%
    \begin{subfigure}{0.5\textwidth}
        \includegraphics[width = \linewidth]{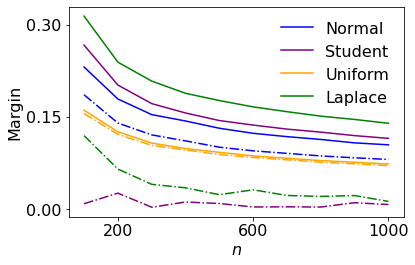}
    \end{subfigure}
    \caption{
    On the left, we consider the same setting as in Figure \ref{figure 1}, however in the case of noiseless data $|\mathcal{O}|=0$.  
    On the right, we plot for noiseless data the margins $\gamma$ of the max-$\ell_1$-margin estimators, as defined in \eqref{def max margin}, as well as the $\ell_1$-margins of AdaBoost $\tilde \beta_T$. }
       \label{figure 2}
\end{figure}

The two plots in Figure \ref{figure 1} show the noisy case, while noise is absent in the two plots in Figure \ref{figure 2}.
For the max-$\ell_1$-margin estimator the prediction error for all simulated features appears to behave identically. By contrast, the $\ell_1$-margin differs widely across the features by a multiplicative constant, but shows the same asymptotic behaviour. 

As stated in Lemma \ref{lemma upper bound iterations}, we see that the margin of AdaBoost is close to the max $\ell_1$-margin and that the performance of AdaBoost is similar to the performance of the max-$\ell_1$-margin classifier. 
The proximity of AdaBoost to its limit appears to depend on the distribution of the features.
In particular, the simulations suggest that heavier tails lead to slower convergence. This is reasonable, considering that AdaBoost rescales the features with their $\ell_\infty$-norm, see Algorithm \ref{alg:main}. 
This is particularly visible when comparing the uniform distribution, for which the max-$\ell_1$-margin estimator and AdaBoost seem to behave almost identically, to the student-t distribution, for which the margin is close to zero for some $n$.

\section{Conclusion}
In this paper, we have shown that AdaBoost, as described in Algorithm \ref{alg:main}, achieves consistent recovery in the presence of small, adversarial errors, despite being overparameterized and interpolating the observations. Our results hold under weak assumptions on the tail behaviour and the behaviour around zero of the feature distribution. 
 In addition, for Gaussian features the derived convergence rates in Corollary \ref{cor rate Gauss studentt} are comparable to convergence rates of state-of-the-art regularized estimators \cite{PlanVershynin13}. This is a first step for the understanding of overparameterized and interpolating AdaBoost and other interpolating algorithms and shows why such algorithms can generalize well in high-dimensional and noisy situations, despite interpolating the data. 

However, in the presence of well-behaved noise, as in logistic regression, our bounds are suboptimal and require that the fraction of mislabeled data points decays to zero. By contrast, regularized estimators  \cite{PlanVershynin13} are able to achieve faster convergence rates in such settings and do not require that the fraction of mislabeled data is asymptotically vanishing to achieve consistency.

Many open question do remain. The convergence rate for Gaussian features in Corollary \ref{cor rate Gauss studentt} is among the best available results if $\beta^*$ is allowed to be genuinely effectively sparse. However, it is not clear whether the exponent in \eqref{cor Euclidean rates Ada} is optimal, and further research about information theoretic lower bounds is needed.  When $\beta^*$ is exactly sparse, the convergence rate in \eqref{cor Euclidean rates Ada} is sub-optimal and better results for log-concave features have been obtained by \cite{AwasthiBalcanHaghtalabZhang16} for a regularized estimator. Moreover, for noiseless data and exact sparse $\beta^*$ our simulations suggest that AdaBoost attains a faster rate than in the noisy case. It thus remains as an interesting further research question how to show that AdaBoost attains faster convergence rates for noiseless data and when $\beta^*$ is exact sparse. 

Finally, our results rely heavily on the anti-concentration assumption in Definition \ref{def anticoncentration}, which is not fulfilled for Rademacher features. Assuming additionally $\|\beta^*\|_\infty=o(1)$ \cite{AiLapanowskiPlanVershynin14} obtained convergence rates for the regularized estimator proposed in \cite{PlanVershynin13}. It is straightforward to adapt our lower bound on the max-$\ell_1$-margin in Theorem \ref{theorem lower bound margin} to such a setting. However, by contrast, it is not clear how to modify the uniform tesselation result used in the proof of Theorem \ref{theorem error bound dependance margin} and consequently show convergence rates for AdaBoost without anti-concentration. 

\section{Proofs} 
\subsection{Proof of Theorem \ref{theorem error bound dependance margin} }
\begin{proof}
 Let $\tilde \beta$ be an approximation of the max $\ell_1$-margin. Defining $\bar \beta := 2\tilde{\beta} /\left( \gamma_0 \| \tilde \beta\|_1 \right)$ we have on an event of probability at least $1-t$ that 
$$
\min_{i \in [n]}  y_i\langle  X_i, \bar \beta \rangle \geq 1 ,
$$
and  $ \bar \beta \in r_n B_1^p$ for $r_n = 2/\gamma_0$. It follows that $\| \tilde \beta \|_1 /\| \tilde \beta \|_2 = \| \bar \beta \|_1 / \| \bar \beta \|_2$, which we bound by applying the following proposition.

\begin{proposition} \label{control norm ratio heavy tail}
Assume $p \gtrsim n$ and that $\mathbb{X}=(X_i)_{i \in [n]}$ has i.i.d. zero mean and unit variance entries and satisfies the weak moment assumption with $\zeta \geq 1/2$.  Let $r_n >0$ be such that 
$$
r_n \lesssim \sqrt\frac{n}{\log^{2\zeta+1}(p)\log(n)}  .
$$
Then, with probability at least $1- cp^{-1} $
for any $\beta \in \mathbb R^p$ such that $\| \beta \|_1 \leq r_n$ and $\min_{i \in [n]}  y_i\langle X_i , \beta \rangle \geq 1$, we have that 
$$
\frac{\| \beta \|_1}{\| \beta \|_2} \lesssim r_n. 
$$
Moreover, assume that $\mathbb{X}$ fulfills a small deviation assumption with parameter $\theta >0$. Then with probability at least $1 - cp^{-1}$, for any $\beta \in \mathbb R^p$ such that $\| \beta \|_1 \leq r_n$ and $\min_{i \in [n]}  y_i \langle  X_i , \beta \rangle \geq 1$, we have that 
$$
\frac{\| \beta \|_1}{\| \beta \|_2} \lesssim \frac{r_n}{\tau_n} ,
$$
where 
$$
\tau_n \asymp  \left[ \frac{ n}{\log^{2\zeta+1}(p)  \log (n)  r_n^2} \right]^{\frac{1}{\theta}}.
$$
\end{proposition}

Having obtained a bound for the ratio $\|\tilde \beta\|_1/\|\tilde \beta\|_2$, we next use a sparse hyperplane tesselation result for the pseudo-metric $d$, arguing by contradiction. Since $d$ is scaling invariant, i.e. $d(\beta,\tilde \beta)=d(\beta,\tilde \beta/\|\tilde \beta\|_2)$ it suffices to consider only elements on the unit sphere.

\begin{proposition} \label{proposition tesselation maurey}
Assume $p\gtrsim n$ and that $\mathbb{X}=(X_i)_{i \in [n]}$ has i.i.d. zero mean and unit variance entries and satisfies the weak moment assumption with $\zeta \geq 1/2$ and the anti-concentration assumption  with 
$\alpha \in (0,1]$. 
For $a >0$, define 
\begin{equation} \label{def eta}
    \eta =c \left(a^2 \frac{\log^{2\zeta+1}(p)\log(n)}{n}   \right)^{\frac{1}{2+\alpha}},
\end{equation}
and assume $\eta \leq 1/2$. Define
$$
 \mathcal B(a,\eta) = \left \{ \beta \in \mathbb R^p :~d(\beta,\beta^*) \geq c \eta^{\alpha} \right\} .
 $$
Then with probability at least $1- c p^{-1}$ we have, uniformly for $\beta \in aB_1^p \cap \mathcal S^{p-1} \cap \mathcal B(a,\eta)$
$$
 \frac{1}{n} \sum_{i=1}^n \mathbf{1} \left ( \sgn{\langle X_i,\beta \rangle } \neq \sgn{\langle X_i,\beta^* \rangle} \right )\gtrsim   \eta^{\alpha} .
$$

\end{proposition}

    Now, we apply Proposition \ref{proposition tesselation maurey} with $a \asymp  r_n$. 
    Since $|\mathcal{O}| \lesssim \eta^\alpha$ by assumption, we get $$
    \frac{1}{n} \sum_{i=1}^n \mathbf{1} \left ( \sgn{\langle X_i,\tilde \beta/\|\tilde \beta\|_2 \rangle } \neq \sgn{\langle X_i,\beta^* \rangle}  \right ) = \frac{|\mathcal{O}|}{n} \lesssim  \eta^\alpha,$$ and hence, adjusting constants, we have on an event of probability at least $1-t-cp^{-1}$ that $\tilde \beta /\|\tilde \beta\|_2 \notin \mathcal{B}(a,\eta)$ and hence on the same event $d(\tilde \beta,\beta^*) \lesssim \eta^\alpha$. \\
    
    When $\mathbb X$ satisfies the small deviation assumption with parameter $\theta>0$, we apply Proposition \ref{proposition tesselation maurey} with $a \asymp r_n/\tau_n$. Since $|\mathcal{O}| \lesssim \eta^{\alpha\left (1+\frac{2}{\theta}\right )}$ by assumption, we conclude the proof using the same reasoning. 
\end{proof}

\subsection{Upper and lower bounds for the max $\ell_1$-margin}

\subsubsection{Proof of Theorem \ref{theorem lower bound margin} } 
We start this section with the following lemma. A proof is given in~\cite{LiangSur20}. 
\begin{lemma} [Proposition A.2 in \cite{LiangSur20}] \label{lemma margin sur}
Suppose that 
$$\gamma:=\max_{\beta \neq 0} \min_{1 \leq i \leq n} \frac{y_i \langle  X_i, \beta \rangle}{\|\beta\|_1} >0 .
$$ Then, we have that $\gamma^{-1} =  \| \hat \beta \|_1$, where 
    \begin{equation} \label{def beta hat}
            \hat \beta \in \argmin_{\beta \in \mathbb R^p} \left \{  \| \beta \|_1 \quad \textnormal{subject to} \quad y_i\langle X_i, \beta \rangle \geq 1 \right \}.
    \end{equation}
\end{lemma}
Hence, in order to lower bound $\gamma$ it suffices to upper bound $\| \hat \beta \|_1$, which is accomplished in the following proposition. 
\begin{proposition} \label{prop upper bound l1 norm}
Assume $p \gtrsim n$ and that $\mathbb{X}=(X_i)_{i \in [n]}$ has i.i.d. symmetric, zero mean and unit variance entries and satisfies the weak moment assumption with $\zeta \geq 1/2$ and the anti-concentration assumption  with
$\alpha \in (0,1]$. 
Then with probability at least $1-cn^{-2}$ we have that

\begin{align} \label{bound l1 norm heavy}
     \|\hat \beta\|_1 \lesssim  \left [ \frac{n}{\log(ep/n)} \left(  s + \frac{\log^{1+2\zeta}(n)| \mathcal O| }{\log(ep/n)} + \log^{1+ 2\zeta}(n)\right)^{\alpha/2} \right ]^{1/(2+\alpha)}. 
 \end{align}
\end{proposition}

\begin{proof}
 We prove Proposition \ref{prop upper bound l1 norm} by explicitly constructing a $\beta$ that fulfills the constraints in \eqref{def beta hat}. 
For $\varepsilon >0$, we define a lifting function  $f_{\varepsilon}: \mathbb R \rightarrow \mathbb R$
$$
f_{\varepsilon}(x) :=  \left\{
    \begin{array}{ll}
        x-\varepsilon & \mbox{if } 0 \leq x \leq \varepsilon \\
        x+\varepsilon & \mbox{if } -\varepsilon \leq x < 0 \\
        0  & \mbox{otherwise. }
    \end{array}
\right.
$$
For $i \in [n]$, we denote $$z_i = \begin{cases} & f_{\varepsilon} \big( \langle X_i, \beta^* \rangle ) ~~~~~~~~~~~~~~~~~~~i \notin \mathcal{O} \\ 
& 2 \langle X_i, \beta^*\rangle-f_{\varepsilon} \big( \langle X_i, \beta^* \rangle ) ~~~i \in \mathcal{O}\\
\end{cases}$$ and $Z = (z_1,\cdots,z_n)^T$. Finally, we define
\begin{align} \label{eq truncation interpolation}
    \hat \nu \in \argmin_{\beta \in \mathbb R^p} \| \beta \|_1 \quad \text{subject to}~~ \langle X_i,\beta \rangle= z_i, ~~~i=1, \dots, n.
\end{align}

By definition of $\hat \nu$, if $i \in \mathcal{O}$, we have the decomposition
\begin{align*}
\langle X_i, \beta^* - \hat \nu \rangle & = -\langle X_i,\beta^* \rangle + f_{\varepsilon} \big( \langle X_i, \beta^* \rangle \big) \\
& =   \left\{
    \begin{array}{ll}
     -\langle X_i, \beta^* \rangle & \mbox{if }  |\langle X_i, \beta \rangle^* | \geq \varepsilon \\
        -\varepsilon & \mbox{if } 0 \leq \langle X_i, \beta^* \rangle \leq \varepsilon \\
         \varepsilon  & \mbox{if } - \varepsilon \leq \langle X_i, \beta^* \rangle < 0.
    \end{array}
\right.
\end{align*}
A similar decomposition with each equation above multiplied with $-1$ holds if $i \notin \mathcal{O}$. Hence, we have that $\sgn{\langle X_i, \beta^*-\hat \nu \rangle}=y_i$ and $|\langle X_i, \beta^*-\hat \nu \rangle|\geq \varepsilon$ for $i=1, \dots,n$. It follows that
\begin{align*}
    \| \hat \beta \|_1 \leq \frac{\|\beta^*-\hat \nu\|_1}{\varepsilon} \leq \frac{\sqrt{s}}{\varepsilon} + \frac{\|\hat \nu \|_1}{\varepsilon}. 
\end{align*}
We now apply Proposition~\ref{quotient property} and obtain that with probability at least $1-2\exp(-2n)$
$$
\|\hat \nu\|_1 \lesssim  \frac{\|  Z \|_2}{\sqrt{\log(ep/n)   }} + \| Z \|_{\infty}. 
$$ 
By Lemma \ref{control infty norm log moments} we have with probability at least $1-n^{-2}$ that
\begin{align*}
    \| Z\|_\infty \leq \varepsilon +  \max_{i \in [n]}|\langle X_i, \beta^* \rangle| \lesssim \varepsilon + \log^{1/2+\zeta}(n). 
\end{align*}
It is left to bound $\| Z\|_2$. By the triangle inequality,  we have that
\begin{align} \label{proof bound Z 1}
    \| Z\|_2 \leq 2 \sqrt{\sum_{i \in \mathcal{O}} |\langle X_i, \beta^* \rangle|^2} + \sqrt{\sum_{i=1}^n f_\varepsilon(\langle X_i, \beta^* \rangle)^2}. 
\end{align}
By Lemma~\ref{control infty norm log moments} we have with probability at least $1-n^{-2}$ 
\begin{align*}
    \sum_{i \in \mathcal{O}} |\langle X_i, \beta^* \rangle|^2 \leq | \mathcal O| \max_{i \in [n]} |\langle X_i, \beta^* \rangle|^2 \lesssim  |\mathcal{O}| \log^{1+2 \zeta}(n). 
\end{align*}
We next bound the second term on the right hand side in \eqref{proof bound Z 1}. Indeed, we have that
 \begin{align}
     \frac{1}{n} \sum_{i=1}^n f_\varepsilon(\langle X_i, \beta^*\rangle)^2 & =\frac{1}{n} \sum_{i=1}^n (|\langle X_i, \beta^*\rangle |-\varepsilon)^2 \mathbf{1}(|\langle X_i, \beta^*\rangle| \leq \varepsilon) \notag \\ 
    & \leq
     \frac{ \varepsilon^2 }{n} \sum_{i=1}^n \mathbf{1}(|\langle X_i, \beta^*\rangle| \leq \varepsilon ) \label{proof bound Z 2}
 \end{align}

Let $p_{\varepsilon} =  \mathbb P (|\langle X_1, \beta^* \rangle | \leq \varepsilon )$. By Hoeffding's inequality, Theorem 3.1.2 in \cite{GineNickl16}, we have with probability at least $1-\exp(-2n\varepsilon^{2\alpha})$ that 
 \begin{align*}
     \frac{1}{n} \sum_{i=1}^n f_\varepsilon(\langle X_i, \beta^*\rangle)^2 \leq \frac{ \varepsilon^2 }{n} \sum_{i=1}^n \mathbf{1}(|\langle X_i, \beta^*\rangle| \leq \varepsilon ) \leq   \varepsilon^2 \left( p_{\varepsilon}+ \varepsilon^{\alpha} \right) \lesssim  \varepsilon^{2+\alpha} , 
 \end{align*}
   where the last inequality holds by the anti-concentration assumption and for $p^{-1} \leq \varepsilon \leq 1$.
 Hence, summarizing, we have with probability at least $1-e^{-2n\varepsilon^{2\alpha}}-n^{-2}- 2\exp(-2n)$ that 
 \begin{align*}
      \| \hat \nu \|_1 \lesssim  \frac{\log^{1/2+ \zeta}(n)| \mathcal O |^{1/2} +n^{1/2} \varepsilon^{1+\alpha/2} }{\sqrt{\log(ep/n)}}   
  +\varepsilon + \log(n)^{1/2+\zeta}
 \end{align*}

Choosing  $$\varepsilon \asymp \left ( \frac{s \log(ep/n)}{n} + \frac{|\mathcal{O}|\log(n)^{1+2 \zeta}}{n } \right )^{\frac{1}{2+\alpha}}$$
concludes the proof.
\end{proof}

\subsubsection{Proof of Proposition \ref{prop margin upper bound}}
\begin{proof}

By the dual formulation of the margin (see Appendix~\ref{dual formulation}), we have that
\begin{align} \label{margin dual}
    \gamma=\inf_{w: ~w_i \geq 0~\forall i \in [n], \|w\|_1=1} \left \|\sum_{i=1}^n w_i y_i X_i \right\|_\infty. 
\end{align}
Hence, for proving an upper bound it suffices to find an appropriate weighting $w$. For $\tau_n$ a sequence to be defined and $\tau_n^{-1}$ taking integer values, we define
 \begin{align*}
  w_i=  \begin{cases}  & \tau_n  ~~~~~i ~\text{is among indices of} ~\tau_n^{-1} ~\text{smallest entries of} (|\langle X_i, \beta^* \rangle |)_{i=1}^n \\
 & 0~~~~~~\text{otherwise}.
    \end{cases}
\end{align*}
We use this choice of $w$ to upper bound $\gamma$. 
We denote the projector onto the space spanned by $\beta^*$ by $P$, $P:=\beta^* (\beta^*)^T$, and define its orthogonal complement $P^\perp:=I_p-P$. 
We have that
\begin{align} \label{proof lower decomp 1}
    \left\| \sum_{i=1}^n w_i y_i X_i \right\|_\infty \leq \left\| \sum_{i=1}^n w_i y_i PX_i \right\|_\infty + \left\| \sum_{i=1}^n w_i y_i P^\perp X_i \right\|_\infty .
\end{align}
We treat the two terms separately.
For the first term, we have by Theorem 5 and Theorem 7 in \cite{GordonLitvakSchuttWerner06} that
\begin{align*}
 \mathbb{E}  \left  \| \sum_{i=1}^n w_i y_i PX_i \right \|_\infty & =   \mathbb{E}\left  \| \sum_{i=1}^n w_i |\langle X_i, \beta^* \rangle| \beta^* \right  \|_\infty =\left  \|\beta^* \right \|_\infty \mathbb{E} \sum_{i=1}^n w_i |\langle X_i, \beta^* \rangle|   \\
 & \lesssim \|\beta^*\|_\infty\sum_{k=1}^{\tau_n^{-1}} \frac{\tau_n k \log(k+1)}{ n}  \lesssim \frac{ \|\beta^*\|_\infty (\tau_n^{-1}+1) \log(p)}{ n }  .
\end{align*}
We next bound the second term on the right hand side in \eqref{proof lower decomp 1}. Observe that $y_i=\text{sgn}(\langle X_i, \beta^*\rangle)=\text{sgn}(\langle PX_i, \beta^* \rangle)$ and hence $y_i$ is independent of $P^\perp X_i$. Likewise, $w$ is a function of $(PX_i)_{i}$ and not $(P^\perp X_i)_i$ and hence $w$ and $P^\perp X_i$ are independent for each $i$. We conclude that \begin{align*}
   \left ( \sum_i w_i y_i P^\perp X_i \right )_j \thicksim \mathcal{N}(0, \|w\|_2^2 \langle e_j, P^\perp e_j\rangle).
\end{align*}
Hence, using a standard Chernoff-bound, we obtain
\begin{align*}
  \mathbb{E}  \left   \| \sum_i w_i y_i PX_i \right  \|_\infty & \lesssim  \sqrt{\log(p)\|w\|_2^2}=\sqrt{\log(p)\tau_n}. 
\end{align*}
Hence, we obtain
\begin{align*}
\mathbb{E} \gamma  \leq \mathbb{E} \left  \|\sum_{i=1}^n w_i y_i X_i \right  \|_\infty \lesssim  \frac{ \|\beta^*\|_\infty \log(p)}{ n\tau_n } + \sqrt{ \log(p)\tau_n}. 
\end{align*}
The final result is obtained by choosing $$\tau_n^{-1}= \left \lceil  \left ( \frac{n}{\|\beta^*\|_\infty \log(p)} \right )^{2/3}\right \rceil. $$
\end{proof}

\subsection{Proof Proposition~\ref{control norm ratio heavy tail}}

\subsubsection{Proof of the first part of Proposition~\ref{control norm ratio heavy tail}
}
In this subsection, we present a result holding only under the weak moment assumption.  We will see in the next section how to improve this result when assuming a small deviation assumption. 
\begin{proposition} \label{control norm ratio heavy tail no small dev}
Assume that $\mathbb{X}=(X_i)_{i \in [n]}$ has i.i.d. zero mean and unit variance entries and satisfies the weak moment assumption with $\zeta \geq 1/2$.  Suppose that $r_n >0$ satisfies
$$
r_n \lesssim  \sqrt\frac{n}{\log(p)}.
$$
 Then, with probability at least 
$$
1- np^{-2} - 2\exp \left( - \frac{cn}{r_n^2\log^{2\zeta}(p)} \right) ,
$$
for any  $\beta \in \mathbb R^p$ such that $\| \beta \|_1 \leq r_n$ and $\min_{i \in [n]}  y_i \langle  X_i , \beta \rangle \geq 1$, we have that $\| \beta \|_2 \geq 1/2$.   
\end{proposition}

\begin{proof}
For $r_n >0$, let $\beta \in \mathbb R^p$ such that $\| \beta \|_1 \leq r_n$ and $\min_{i \in [n]} y_i \langle X_i , \beta \rangle \geq 1$. Thus, we have
\begin{equation} \label{proof absurd final}
    \frac{1}{n} \sum_{i=1}^n |\langle X_i, \beta \rangle|  \geq 1 .
\end{equation}
We proceed by contradiction. Assume that $\| \beta \|_2 \leq 1/2$. In this case, we show that Equation~\eqref{proof absurd final} is not satisfied with large probability, concluding the proof by contradiction.  \\
For $i \in [n]$, using H\"older's inequality, we have that 
$$
| \langle X_i , \beta \rangle|  \leq \| X_i \|_{\infty} \| \beta \|_1 \leq r_n \| X_i \|_{\infty} \lesssim r_n \log^{\zeta}(p) ,
$$
where the last inequality follows from Lemma~\ref{control infty norm log moments} and holds with probability at least $1-p^{-2}$. Thus, with probability at least $1-n/p^2$ we have, for all $i \in [n]$, that $| \langle X_i , \beta \rangle|   \lesssim r_n  \log^{\zeta}(p)$. Hence, conditioning on this event and using the bounded differences inequality, Theorem 3.3.14 in \cite{GineNickl16}, we obtain with probability at least 
$$
1-2\exp \left( - \frac{cn}{ r_n^2 \log^{2 \zeta}(p)} \right) - n/p^2$$ 
that we have
\begin{align*}
    \sup_{\beta \in r_n B_1^p \cap (1/2)B_2^p}  \frac{1}{n} \sum_{i=1}^n |\langle X_i, \beta \rangle| & \leq \sup_{\beta \in r_n B_1^p \cap (1/2)B_2^p} \mathbb E | \langle X_1, \beta \rangle | \\
    & + \mathbb E \sup_{\beta \in r_n B_1^p \cap (1/2)B_2^p}  \frac{1}{n} \sum_{i=1}^n |\langle X_i, \beta \rangle| - \mathbb E | \langle X_i, \beta \rangle |  + \frac{1}{4}.
\end{align*}
By Jensen's inequality and the fact  $X$ is isotropic with unit variance, we obtain that 
$$
\sup_{\beta \in r_n B_1^p \cap (1/2) B_2^p} \mathbb E | \langle X_1, \beta \rangle | \leq 1/2 . 
$$
Moreover, we have that 
\begin{align*}
     \mathbb E \sup_{\beta \in r_n B_1^p \cap (1/2)B_2^p}  \frac{1}{n} \sum_{i=1}^n |\langle X_i, \beta \rangle| - \mathbb E | \langle X_i, \beta \rangle |  & \leq  \mathbb E \sup_{\beta \in r_n B_1^p \cap (1/2)B_2^p}  \frac{4}{n} \sum_{i=1}^n \sigma_i \langle X_i, \beta \rangle \\
     & \lesssim  r_n \sqrt \frac{\log(p)}{n},
\end{align*}
where $(\sigma_i)_{i=1}^n$ are i.i.d Rademacher random variables independent from $(X_i)_{i=1}^n$. We used in the first line the symmetrization and contraction principles, Theorem 3.1.21 and Theorem 3.2.1. in \cite{GineNickl16} and Proposition~\ref{control Rademacher complexity low moments} in the second line to bound the Rademacher complexity. \\
The condition on $r_n$ shows that
$$
   \sup_{\beta \in r_n B_1^p \cap (1/2)B_2^p}  \frac{1}{n} \sum_{i=1}^n |\langle X_i, \beta \rangle|<1, 
$$
and the contradiction is established. 

\end{proof}

\subsubsection{Proof of the second part of Proposition~\ref{control norm ratio heavy tail}: small deviation assumption}

In this subsection, we show how to prove the second part of Proposition~\ref{control norm ratio heavy tail} under the small deviation assumption~\ref{def small deviation}. 

\begin{proposition} \label{minimum.theorem} 
Assume $p \gtrsim n$ and that $\mathbb{X}=(X_i)_{i \in [n]}$ has i.i.d. zero mean and unit variance entries and satisfies the weak moment assumption with $\zeta \geq 1/2$.
Moreover, assume that $\mathbb{X}$ fulfills a small deviation assumption, Definition~\ref{def small deviation}, with constant $\theta >0$.  Let $r_n \geq 1$,
define
$$ \tau_n =c \left ( \frac{n}{\log^{2\zeta+1}(p)\log(n)r_n^2}\right )^{\frac{1}{\theta}}
$$
and suppose that $\tau_n \gtrsim 1 $. 
 Then, with probability at least $1-p^{-1}$ for any $\beta \in \mathbb{R}^p$ such that $\|\beta\|_1 \leq r_n$ and $\min_{i \in [n]} y_i \langle X_i, \beta \rangle \geq 1 $, we have that $\| \beta\|_1 / \|\beta\|_2 \lesssim r_n/\tau_n$. 
\end{proposition}

{\bf Proof of Proposition~\ref{minimum.theorem}.}
Let $\{ {\rm e}_j \}_{j=1}^p$ be the set of standard unit vectors in $\mathbb{R}^p$ and
$  {\cal D} := \{ \pm e_j \} \cup \{ 0 \}  \subset \mathbb{R}^p$ be the set of vectors with all entries equal to zero
possibly except just one, where the value is then $\pm 1$.
We define, for $m \in \mathbb{N}$, Maurey's set
$$ {\cal Z}_m := \biggl \{ z = {1 \over m} \sum_{k=1}^m z_k  , \ z_k \in  {\cal D} \ \forall \ k \biggr \}. $$ 
Take
 $$ m =c\log^{2\zeta}(p) \log(n) r_n^2  , $$
 and define the event
 $${\cal E}_{\rm max}:= \{ \|\mathbb X\|_\infty \leq c  \log^{\zeta}(p) \}, $$
 and observe that by Lemma~\ref{control infty norm log moments}, the event ${\cal E}_{\rm max}$ occurs with probability at least $1-p^{-1}$ as $p \gtrsim n$.  
 Then, by Lemma \ref{Maurey.lemma}, 
for all $\beta \in r_n B_1^p$ such that $\| \beta \|_2 \lesssim \tau_n$ there exists a vector $z_{\beta } \in r_n {\cal Z}_m $ such that on ${\cal E}_{\rm max}$
$$\max_{1 \le i \le n } | \langle X_i , \beta \rangle - \langle X_i , z_{\beta}  \rangle
  | \lesssim  \log^{\zeta}(p) r_n
  \sqrt { \log (2n) \over m }   \le {1 \over 2} $$
  as well as $\| \beta - z_\beta \|_2 \le 1/2 $, for $m$ defined previously.
  Thus, we also have by assumption on $\tau_n$
   $$ \| z_{\beta } \|_2 \lesssim \tau_n + 1/2 \lesssim \tau_n . $$
   In other words, on ${\cal E}_{\rm max}$ we have that $ \{ z_{\beta } :  \| \beta \|_1 \leq r_n, \|\beta \|_2 \leq c \tau_n \} \subset
 r_n {\cal Z}_m  \cap \{ \beta : \| \beta \|_2 \leq c \tau_n \}  =: {\cal Z}_m (r_n , \tau_n) $. We invoke that
 \begin{eqnarray*}
  & &\biggl \{ \sup_{\beta \in r_n B_1^p  \cap \{ \beta : \| \beta\|_2 \lesssim \tau_n \} } \min_{1 \le i \le n } | \langle X_i , \beta \rangle | \ge 1  \biggr \} \cap {\cal E}_{\rm max}  \\
 & \subseteq  & \biggl \{ \max_{z \in {\cal Z}_m 
 (r_n, \tau_n)}  \min_{1 \le i \le n } | \langle X_i , z \rangle | \ge {1 \over 2} \biggr \}  . 
  \end{eqnarray*}
   For all $z \in {\cal Z}_m (r_n ,  \tau_n) $ and $i \in [n]$ by the small deviation assumption~\ref{def small deviation}, we have
  $$ \mathbb{P}\biggl  ( |  \langle X_i , z \rangle | \le { 1 \over 2} \biggr ) \ge \mathbb{P} \biggl ( {| \langle X_i , z \rangle|\over  \| z \|_2  } \leq {c \over   \tau_n} 
  \biggr ) \gtrsim \tau_n^{-\theta} . $$
  Hence,
  $$ \mathbb{P}\biggl  ( | \langle X_i , z \rangle | \ge { 1 \over 2} \biggr ) \leq \biggl (1-c\tau_n^{-\theta}  \biggr  ) \leq
  \exp\biggl [- c\tau_n^{-\theta} \biggr ]  $$
  and thus we obtain
  $$ \mathbb{P}\biggl  ( \min_{1 \le i \le n } | \langle X_i , z \rangle |  \ge { 1 \over 2} 
 \biggr ) \leq    \exp\biggl [-cn \tau_n^{-\theta} \biggr ] .$$
  
 Since
  $$| {\cal Z}_m (r_n ,\tau_n)  | \le  | {\cal Z}_m | \le ( 2p+1)^m . $$
  we obtain by a union bound that 
  $$ \mathbb{P}\biggl  ( \max_{ z \in {\cal Z}_m ( r_n ,  \tau_n)} \min_{1 \le i \le n } | \langle X_i , z \rangle| \ge { 1 \over 2} \biggr ) \leq   \exp\biggl [ m \log (2p+1 ) - cn \tau_n^{-\theta} \biggr ].$$

 We conclude that 
\begin{align*}\mathbb{P} \biggl ( \sup_{\beta \in r_n B_1^p \cap \{ \beta: \|\beta\|_2 \leq c \tau_n \} }
 \min_{1 \le i \le n } |  \langle X_i , \beta \rangle| \ge 1 \biggr ) & \leq
 \exp\biggl [ m \log (2p+1 ) -  cn \tau_n^{-\theta} \biggr ]  + \mathbb{P} ( {\cal E}_{\rm max}^c ) \\
 & \leq \exp \left(- c n  \tau_n^{-\theta} \right) + p^{-1},
 \end{align*}
from our choice of $m$ and applying Lemma~\ref{control infty norm log moments}.

\subsection{Tesselation}

\begin{proposition} \label{proposition tesselation maurey}
Assume $p \gtrsim n$ and that $\mathbb{X}=(X_i)_{i \in [n]}$ has i.i.d. zero mean and unit variance entries and satisfies the weak moment assumption with $\zeta \geq 1/2$ and the anti-concentration assumption  with 
$\alpha \in (0,1]$. 
For $a >0$ define 
\begin{equation} \label{def eta}
    \eta =c  \left( \frac{a^2 \log^{2\zeta+1}(p)\log(n)}{n}   \right)^{\frac{1}{2+\alpha}},
\end{equation}
and assume  $\eta \lesssim 1$. Define
$$
 \mathcal B(a,\eta) = \left \{ \beta \in \mathbb R^p : d(\beta,\beta^*) \geq c \eta^{\alpha} \right\} .
 $$
Then with probability at least
$$
1 - 2\exp \left( -c n \eta^\alpha  \right) - np^{-2}
$$
we have, uniformly for $\beta \in aB_1^p \cap \mathcal S^{p-1} \cap \mathcal B(a,\eta)$
$$
 \frac{1}{n} \sum_{i=1}^n \mathbf{1} \{  \sgn{\langle X_i,\beta \rangle } \neq \sgn{\langle X_i,\beta^* \rangle } \} \gtrsim \eta^{\alpha} .
$$
\end{proposition}

\begin{proof}
For $a>0$ and $\eta$ defined in Equation~\eqref{def eta} let $\beta \in aB_1^p \cap \mathcal{S}^{p-1} \cap \mathcal B(a,\eta)$.
By Lemma~\ref{Maurey.lemma} 
there exists $z_{\beta}$ in $\mathcal Z_m$ such that
$$
\max_{i \in [n]} |\langle X_i, \beta - z_{\beta} \rangle |{\rm l }_{{\cal E}_{\rm max}}  \lesssim a \log^{\zeta}(p)\sqrt \frac{\log(n)}{m}  \asymp \eta~~\text{and}~~ \| \beta - z_{\beta}\|_2 \lesssim \frac{1}{\sqrt{m}}
$$
where ${\cal E}_{\rm max}:=\{ \|\mathbb{X}\|_\infty \leq c\log^{\zeta}(p)\}$ and
for $m =ca^2\log^{2\zeta}(p)\log(n)/\eta^2$. We note that by Lemma \ref{control infty norm log moments} ${\cal E}_{\rm max}$ occurs with probability at least $1-np^{-2}$. In particular we have $1/2\leq 1-\eta \leq\| z_{\beta} \|_2 \leq 1+ \eta \leq 3/2$, for $\eta$ small enough. 
Let $z_{\beta} \in \mathcal Z_m$. By Bernstein's inequality, Theorem 3.1.7 in \cite{GineNickl16}, and the anti-concentration assumption, we have that 
\begin{align*}
   \sum_{i=1}^n \mathbf{1} \{ | \langle X_i, z_{\beta} \rangle | \leq  \eta \} &  \leq n \left( \mathbb P \left(| \langle X_1, z_{\beta} \rangle | \leq  \eta \right)  +  \eta^{\alpha} \right) \\
   & \leq n \left( \mathbb P \left(| \langle X_1, \frac{z_{\beta}}{\| z_{\beta}\|_2} \rangle | \leq  2\eta \right) + \eta^{\alpha} \right) \\
   & \leq n \left(   \sup_{\beta \in \mathcal S^{p-1}} \mathbb P \left(  \langle X_1, \beta \rangle | \leq 2\eta \right)  +   \eta^{\alpha} \right) \\
   & \lesssim  n \eta^{\alpha}
\end{align*}
with probability at least  $1- \exp(- cn\eta^{\alpha})$.
Now, define
$$
J := \left \{ i \in [n] : \min_{z_\beta \in \mathcal Z_m} |\langle X_i, z_{\beta} \rangle| \geq \eta \right \}. 
$$
Using an bound over $\mathcal Z_m$ and that $|\mathcal Z_m|\leq (2p+1)^m$, we obtain that with probability at least 
$$
1 - 2\exp \left[  m\log(2p+1) - cn\eta^\alpha   \right] \geq 1- 2\exp[- cn \eta^\alpha]
$$
we have uniformly for $z_\beta \in \mathcal{Z}_m$
\begin{equation*}
    |J^C| \lesssim   \eta^{\alpha}n .
\end{equation*}
For $i \in J$ and working on the event ${\cal E}_{\rm max}$ we have that $|\langle X_i, z_\beta \rangle| \geq \eta$ and $| \langle X_i,\beta-z_\beta \rangle| < \eta$ and hence $\langle X_i,\beta \rangle$ and $\langle X_i,z_\beta \rangle$ have matching signs. 

Hence, for $\beta \in aB_1^p \cap \mathcal S^{p-1} \cap \mathcal B(a,\eta)$ and working on the event ${\cal E}_{\rm max}$, we have 
\begin{align*}
    \sum_{i=1}^n \mathbf{1} \{  \sgn{\langle X_i, \beta \rangle} \neq \sgn{\langle X_i,\beta^* \rangle} \} & \geq \sum_{i \in J} \mathbf{1} \{  \sgn{\langle X_i,\beta \rangle} \neq \sgn{\langle X_i,\beta^* \rangle} \} \\ 
    & = \sum_{i \in J} \mathbf{1} \{  \sgn{\langle  X_i, z_{\beta} \rangle} \neq \sgn{\langle X_i,\beta^* \rangle} \} \\
    & \geq \sum_{i =1}^n \left (  \mathbf{1} \{  \sgn{\langle X_i, z_{\beta} \rangle} \neq \sgn{\langle X_i,\beta^* \rangle} -  c \eta^{\alpha} \right ). 
\end{align*}
Applying Bernstein's inequality, Theorem 3.1.7 in \cite{GineNickl16} we have that
\begin{align*}
      \sum_{i =1}^n \mathbf{1} \left ( \sgn{\langle  X_i, z_{\beta} \rangle} \neq \sgn{\langle  X_i,\beta^* \rangle} \right ) \geq &  \bigg( d(z_\beta,\beta^*) -  c\eta^{\alpha}   \bigg) n
\end{align*}
with probability at least  $1- \exp \left( - c n \eta^{\alpha}\right)$. We next lower bound $d(z_\beta,\beta^*)$. Indeed, arguing as above, we have that
\begin{align*}
&   d(z_\beta,\beta^*) =  \mathbb{P} \left (  \sgn{\langle X,z_\beta \rangle} \neq \sgn{\langle X, \beta^* \rangle}| (y_i,X_i)_{i=1}^n \right) \\
& \geq  \mathbb{P} \left (  \sgn{\langle X,z_\beta \rangle} \neq \sgn{\langle X, \beta^* \rangle}, ~|\langle X, z_\beta \rangle| \geq \eta, ~|\langle X, z_\beta-\beta \rangle | < \eta | (y_i,X_i)_{i=1}^n \right ) \\
 & = \mathbb{P} \left (  \sgn{\langle X,\beta\rangle} \neq \sgn{\langle X, \beta^* \rangle}, ~|\langle X, z_\beta \rangle| \geq \eta, ~|\langle X, z_\beta-\beta \rangle | < \eta | (y_i,X_i)_{i=1}^n \right ) \\
 & \geq d(\beta,\beta^*)-\mathbb{P} \left ( |\langle X, z_\beta \rangle| \leq \eta | (y_i,X_i)_{i=1}^n \right ) - \mathbb{P} \left ( |\langle X, z_\beta -\beta\rangle| \geq  \eta  | (y_i,X_i)_{i=1}^n \right ) .
\end{align*}
Since  $d(\beta,\beta^*) \gtrsim \eta^{\alpha}$, 
$\mathbb{P} \left ( |\langle X, z_\beta \rangle| \leq \eta | (y_i,X_i)_{i=1}^n \right ) \lesssim \eta^{\alpha}$ by the anti-concentration assumption (Definition~\ref{def anticoncentration}) and $\mathbb{P} \left ( |\langle X, z_\beta -\beta\rangle| > \eta | (y_i,X_i)_{i=1}^n \right ) \lesssim n^{-2} \lesssim \eta^{\alpha}$ by our choice of $m$ and Lemma \ref{control infty norm log moments}, we obtain when the constant in the definition of $\mathcal{B}(a,\eta)$ is large enough that 
$$
 d(z_\beta,\beta^*) \gtrsim \eta^{\alpha} . 
$$

Hence, taking another union bound over $\mathcal Z_m$ and ${\cal E}_{\rm max}$ and for the constant in the definition of $\mathcal{B}(a,\eta)$ large enough, we obtain with probability at least 
$$
1 - 2\exp \left[  m\log(2p+1) - c\eta^\alpha n \right] -np^{-2} \geq 1-2\exp[- c \eta^\alpha n ]-np^{-2}
$$
that  uniformly for $\beta \in aB_1^p \cap \mathcal S^{p-1} \cap \mathcal B(a,\eta)$ 
$$
  \sum_{i=1}^n \mathbf{1}\left (  \sgn{\langle X_i,\beta \rangle} \neq \sgn{\langle X_i, \beta^* \rangle} \right )\gtrsim \eta^\alpha n 
$$
which concludes the proof. 

\end{proof}

\subsection{Rest of the proofs}
\subsubsection{Lemma \ref{Maurey.lemma}}
The following Lemma applies Maurey's empirical method \cite{Carl85, GuedonLecuePajor13} to construct a set $\mathcal{Z}_m$ that approximates the $B_1^p$-ball well. 
 \begin{lemma}\label{Maurey.lemma} (Maurey's Lemma) Let $\{ {\rm e}_j \}_{j=1}^p$ be the set of standard unit vectors in $\mathbb{R}^p$ and
$  {\cal D} := \{ \pm e_j \} \cup \{ 0 \}  \subset \mathbb{R}^p$ be the set of vectors with all entries equal to zero
except at most one, where the value is then $\pm 1$.
Define, for $m \in \mathbb{N}$, Maurey's set
$$ {\cal Z}_m := \biggl \{ z = {1 \over m} \sum_{k=1}^m z_k  , \ z_k \in  {\cal D} \ \forall \ k \biggr \}. $$ 
  Then, we have that ${\cal Z}_m \subset B_1^p $ and that $|{\cal Z}_m |\le (2p+1)^m$. Moreover, 
  for every $\beta \in B_1^p$ there exists a vector $z_{\beta} \in
  {\cal Z}_m $ such that for ${\cal E}_{\rm max}:=\{ \|\mathbb{X}\|_\infty \leq c \log^{\zeta}(p) \} $ we have that 
 $$ \max_{1 \le i \le n } | \langle X_i , \beta \rangle - \langle X_i , z_{\beta} \rangle|{\rm l }_{{\cal E}_{\rm max}} 
 \lesssim
\log^{\zeta}(p) \sqrt { \log (n) \over m } ~~~~\text{and}~~~~\| \beta - z_{\beta } \|_2  \lesssim \frac{1}{\sqrt{m}}.$$
  \end{lemma}
  \begin{proof}

  For $z \in {\cal D}$ either $\| z \|_1 =1 $ or $z \equiv 0 $.
  Thus for $\bar z := \sum_{k=1}^m z_k / m \in {\cal Z}_m$ we have
  $\| \bar z \|_1 \le \sum_{k=1}^m \| z_k \|_1 / m \le 1 $. 
  It is moreover clear that $|{\cal D}| = (2p+1)$. Therefore $|{\cal Z}_m |\le (2p+1)^m$. 
  
  We now turn to the main part of the lemma.
 Let $\beta \in B_1^p$. Define a random vector $Z \in {\cal D}$ by
 $$ \mathbb{P} \biggl ( Z =  {\rm sign} (\beta_j)  {\rm e}_j  \biggr ) = | \beta_j | , \ {\rm for} \ \beta_j \not= 0, \  j=1 , \ldots , p ,$$
 and 
 $$ \mathbb{P}\biggl  ( Z = 0  \biggr )= 1- \| \beta \|_1 .$$
 Then
 $$ \mathbb{E} Z = \beta , \ \mathbb{E} \| \beta -Z \|_2^2 = \| \beta \|_1 - \|\beta \|_2^2 \le \| \beta \|_1 \le 1 . $$
 Let $Z_1 , \ldots ,Z_m$ be independent copies of $Z$ and define $\bar Z :=\sum_{k=1}^m Z_k /m$.
 Then we get
 $$ \mathbb{E} \| \beta - \bar Z \|_2^2 \le {1 \over m} .$$
 Let $\sigma_1 , \ldots \sigma_m $ be a Rademacher sequence  independent of ($\mathbb{X}, 
 (Z_1 , \ldots , Z_m))$.
 Then we have by the symmetrization inequality, Theorem 3.1.21  in \cite{GineNickl16}, that 
 $$ \mathbb{E} \biggl [ \max_{1 \le i \le n } | \langle X_i , \beta \rangle  - \langle X_i ,\bar  Z \rangle |\  \biggr \vert \mathbb{X} \biggr ] \le 
{2 \over m}  \mathbb{E} \biggl [ \max_{1 \le i \le n } | \sum_{k=1}^m \sigma_k \langle X_i ,\bar  Z_k \rangle | \ 
 \biggr  \vert \ \mathbb{X}  \biggr ] .$$
Further, for $i=1 , \ldots , n $, we have that
$$  \sum_{k=1}^m \langle X_i , Z_k \rangle^2 \le m  \| X_i \|_{\infty}^2 \le m  \| \mathbb{X} \|_{\infty}^2 .$$
Thus we obtain, 
$$\mathbb{E} \biggl [ \max_{1 \le i \le n } |\sum_{k=1}^m \sigma_k \langle X_i , Z_k \rangle |\  \biggr \vert \ \mathbb{X}, \ Z_1 , \ldots , Z_m   \biggr ] \le
\sqrt { 2 \log (2n) } \sqrt m  \| \mathbb{X} \|_{\infty} . $$
Hence, and since
$${\cal E}_{\rm max} = \biggl \{ \| \mathbb{X} \|_{\infty} \leq c \log^{\zeta}(p) \biggr \}, $$
we obtain that 
 $$ \mathbb{E} \biggl [ \max_{1 \le i \le n } | \langle X_i , \beta \rangle - \langle X_i , \bar Z \rangle
  | {\rm l }_{{\cal E}_{\rm max}}   \biggr ] \lesssim  
  \log^{\zeta}(p) \sqrt{ \frac{\log(n)}{m}}.$$
Invoking Jensen's inequality and $ {\mathbb{E} \| \beta - \bar Z \|_2^2 }\leq 1/m$ 
 we have that 
 \begin{align*}
     \mathbb{E} \|\beta-\bar Z\|_2 \lesssim 1/\sqrt{m}.
 \end{align*}
 Hence we obtain that
 \begin{align*}
       \mathbb{E} \left [ \max_{1 \le i \le n } | \langle X_i , \beta \rangle - \langle X_i , \bar Z \rangle
  | {\rm l }_{{\cal E}_{\rm max}} + \log^{\zeta}(p)\log^{1/2}(n) \|\beta-\bar Z\|_2 \right ] \lesssim \log^{\zeta}(p) \sqrt{\frac{\log(n)}{m}},
 \end{align*}
  and hence there exists at least one $z_\beta \in \mathcal{Z}_m$ with the desired properties. 
  \end{proof}
  
\subsubsection{Proof of Lemma~\ref{lemma upper bound iterations}}

\begin{proof}
The proof follows closely the arguments in \cite{Telgarsky13}. First, note that rescaling $\mathbb{X}=\mathbb{X}/\|\mathbb{X}\|_\infty$ does not change the approximating properties of $\tilde \beta_T$ for the max $\ell_1$-margin. Indeed, if $\tilde \beta_T$ fulfills
\begin{align*}
    \min_{1 \leq i \leq n} \frac{y_i \left \langle \frac{X_i}{\|\mathbb{X}\|_\infty}, \tilde \beta_T \right \rangle}{\|\tilde \beta_T\|_1} \geq \frac{1}{2} \max_{\beta \neq 0} \min_{1 \leq i \leq n} \frac{y_i \left \langle \frac{X_i}{\|\mathbb{X}\|_\infty}, \beta \right \rangle}{\|\beta\|_1}=:\gamma_R=\gamma/\|\mathbb{X}\|_\infty,
\end{align*}
then, by linearity, $\tilde \beta_T$ also fulfills \begin{align*}
    \min_{1 \leq i \leq n} \frac{y_i \langle X_i, \tilde \beta_T \rangle}{\|\tilde \beta_T\|_1} \geq \frac{1}{2} \max_{\beta \neq 0} \min_{1 \leq i \leq n} \frac{y_i \langle X_i, \beta\rangle}{\|\beta\|_1}=\gamma.
\end{align*}
Henceforth, we work with the rescaled data $\mathbb{X}/\|\mathbb{X}\|_\infty$, which, in slight abuse of notation, we also denote by $\mathbb{X}$. 
Note, that by definition $\|\mathbb{X}\|_\infty \leq 1$. 
Define the exponential loss,
$$\ell(\beta):=\frac{1}{n}\sum_{i=1}^n \exp(-y_i \langle X_i, \beta \rangle).$$ 
Note that
\begin{align*}& \nabla \ell( \beta) = -\frac{1}{n}\sum_{i=1}^n y_i X_i \exp(-y_i\langle X_i, \beta \rangle)  ~~\text{and}\\&\nabla^2 \ell(\beta)=\frac{1}{n}\sum_{i=1}^n X_i X_i^T \exp(-y_i \langle X_i, \beta \rangle ). 
\end{align*}
Hence, we have that 
\begin{align*}
{-\langle \nabla \ell(\tilde \beta_t ), v_t \rangle}
= \alpha_t \ell(\tilde \beta_t). 
\end{align*}
Moreover, note that $|\alpha_t|\leq \sum w_{t,i} |\langle X_i, v_t \rangle| \leq 1$. 
By second order Taylor expansion, we obtain that 
\begin{align*}
    \ell(\tilde \beta_{t+1}) & \leq \ell(\tilde \beta_t)+\epsilon \alpha_t \langle \nabla \ell(\tilde \beta_t), v_t \rangle + \frac{1}{2} 
    \sup_{r \in [0, 1]} \langle \nabla^2 \ell(\tilde \beta_t+r\epsilon \alpha_t v_t) v_t, v_t \rangle. 
    \end{align*}
    We next bound the Hessian above. Indeed, we have for any $r$ that
    \begin{align*}
        \langle \nabla^2 \ell(\tilde \beta_t + r \epsilon \alpha_t v_t) v_t, v_t \rangle & = \frac{1}{n} \sum_{i=1}^n \langle X_i, v_t \rangle^2 \epsilon^2 \alpha_t^2 \exp(-y_i \langle X_i, \tilde \beta_t + r \epsilon \alpha_t v_t \rangle )\\ &  \leq \epsilon^2 \alpha_t^2 \exp(r | \alpha_t|  \epsilon  ) \ell(\tilde \beta_t) \leq \epsilon^2 \alpha_t^2 e^{\epsilon} \ell(\tilde \beta_t). 
    \end{align*}
    Hence, we can further bound
    \begin{align*}
    \ell(\tilde \beta_{t+1})
   &  \overset{}{\leq} \ell(\tilde \beta_t)+\epsilon \alpha_t \langle \nabla \ell(\tilde \beta_t), v_t \rangle+\frac{\epsilon^2 \alpha_t^2 e^{\epsilon} }{2}\ell(\tilde \beta_t)  \\
   & \leq \ell(\tilde \beta_t) \left (1-\epsilon \alpha_t^2+\frac{3\epsilon^2 \alpha_t^2  }{2}  \right ) \leq \ell(\tilde \beta_t)  \exp \left (-\epsilon \left ( \alpha_t^2-\frac{3\epsilon \alpha_t^2 }{2}  \right ) \right ),
\end{align*}
and hence we obtain
\begin{align*}
      \ell(\tilde \beta_{T}) \leq \exp \left (-\epsilon \sum_{t=1}^T \left (  \alpha_t^2-\frac{3\epsilon \alpha_t^2  }{2} \right ) \right ).
\end{align*}
Moreover, we have that
\begin{align*}
    \|\tilde \beta_T\|_1 =\|\sum_{t=1}^T \epsilon \alpha_t v_t \|_1 \leq \epsilon \sum_{t=1}^T |\alpha_t|.
\end{align*}
In addition, we note that by the dual formulation of the margin (see Appendix~\ref{dual formulation}) and definition of $v_t $ and $\alpha_t$ we have that
$$|\alpha_t|=\|\sum w_{t,i} y_i X_i\|_\infty \geq \inf_{w:~w_i \geq 0~\forall i, \| w\|_1=1} \|\sum w_i y_i X_i\|_\infty=\gamma_R.$$
Hence, by Markov's inequality and since $3 \epsilon/2 < 1$, we obtain for any positive $x$
\begin{align*}
    \sum_{i=1}^n \mathbf{1}_{\{  y_i \langle X_i, \tilde \beta_T \rangle   \leq \|\tilde \beta_T \|_1 x \} } & \leq \sum_{i=1}^n \exp(\|\tilde \beta_T\|_1 x - y_i \langle X_i, \tilde \beta_T \rangle) \\
    & = n\ell(\tilde \beta_T) \exp(\|\tilde \beta_T\|_1 x )\\ & \leq \exp \left (\log(n)-\epsilon \sum_{t=1}^T |\alpha_t| \left (|\alpha_t|-x-\frac{3\epsilon |\alpha_t|  }{2} \right ) \right ) \\
    & \leq \exp \left (\log(n)-\epsilon \sum_{t=1}^T |\alpha_t| \left (\gamma_R-x-\frac{3\epsilon \gamma_R  }{2} \right ) \right ).
\end{align*}
Hence,  choosing $x=\frac{1}{2}\gamma_R$ and using that $\epsilon \leq 1/6$ and that  with probability at least $1-np^{-2}$ we have by Lemma \ref{control infty norm log moments} that $T  > \frac{2\log(n)}{3 \epsilon^2 \gamma_R^2} =\frac{2\log(n) \|X\|_\infty^2}{3 \epsilon^2 \gamma^2}$,
 we obtain
\begin{align*}
    \sum_{i=1}^n \mathbf{1}_{\{  y_i \langle X_i, \tilde \beta_T \rangle   \leq \|\tilde \beta_T \|_1 x \} } & \leq \sum_{i=1}^n \exp(\|\tilde \beta_T\|_1 x - y_i \langle X_i, \tilde \beta_T \rangle) \\ & \leq \exp \left (\log(n)-3T\epsilon^2 \gamma_R^2 /2 \right ) < e^{0}=1.
\end{align*}
Since $   \sum_{i=1}^n \mathbf{1}_{\{  y_i \langle X_i, \tilde \beta_T \rangle \leq \|\tilde \beta_T\|_1 \frac{1}{2}\gamma_R \}}$ can only take values in $\{0, 1, \dots, n\}$ this implies that $\sum_{i=1}^n \mathbf{1}_{\{  y_i \langle X_i, \tilde \beta_T \rangle \leq \|\tilde \beta_T\|_1 \frac{1}{2}\gamma_R \}}=0$ and hence the result follows.
\end{proof}

\subsubsection{Proof of Corollary \ref{cor rate Gauss studentt}} 
\begin{proof} 
For Gaussian distributed features it is clear that the weak moment assumption with $\zeta=1/2$ is satisfied. Moreover, since for $\beta \in \mathcal{S}^{p-1}$ we have that $\langle X, \beta \rangle \thicksim \mathcal{N}(0,1)$ we have for any $0<\varepsilon \leq 1$
\begin{align*}
   \sup_{\beta \in \mathcal{S}^{p-1}} \mathbb{P} \left ( | \langle X, \beta \rangle | \leq \varepsilon \right ) = \int_{-\varepsilon}^\varepsilon \frac{1}{\sqrt{2 \pi}} e^{-x^2/2} \text{d}x \asymp \varepsilon,
\end{align*}
and hence both the anti-concentration and small deviation assumptions are fulfilled with $\alpha=\theta=1$. Finally, to show \eqref{cor Euclidean rates Ada}, note that by Grothendieck's identity, Lemma 3.6.6. in \cite{vershynin2018high}, and as the geodesic distance on the sphere is lower bounded by the Euclidean distance, we have that
\begin{align*}
d(\beta^*, \tilde \beta_T) =  \frac{\arccos{\left ( \left \langle \beta^*, \frac{\tilde \beta_T}{\|\tilde \beta_T\|_2}\right \rangle \right )}}{\pi} \geq \left \| \beta^*-\frac{\tilde \beta_T}{\|\tilde \beta_T\|_2}\right \|_2. 
\end{align*}
For the student-t-distribution with at least $32 \log(p)$ degrees of freedom Lemma \ref{lemma studentt conditions} below proves that the weak moment assumption and the anti-concentration and small deviation assumptions with $\alpha=\theta=1$ are satisfied. Moreover, Lemma \ref{Lemma margin studentt} below, shows that in this case we can also lower bound $d(\tilde \beta_T,\beta^*) \gtrsim \left \| \beta^*-\frac{\tilde \beta_T}{\|\tilde \beta_T\|_2}\right \|_2.$
\end{proof}
\begin{lemma} \label{lemma studentt conditions}
Suppose that $X=(x_j)_{j=1}^p$ with $x_j \overset{i.i.d.}{\thicksim} \sqrt{(d-2)/d}t_d$ for $d \in \mathbb{N}$, $d \geq 32 \log(p)$ and $p \gtrsim 1$. Then for any $0 \leq \varepsilon\leq 1$ 
\begin{align*}
    \inf_{\beta \in \mathcal{S}^{p-1}} \mathbb{P} \left ( |\langle X,\beta \rangle|\leq \varepsilon \right )\gtrsim \varepsilon. 
\end{align*}
Moreover, under the same assumptions, we have that
\begin{align*}
    \sup_{\beta \in \mathcal{S}^{p-1}} \mathbb{P}  \left ( |\langle X,\beta \rangle|\leq \varepsilon \right ) \lesssim \varepsilon +p^{-1}. 
\end{align*}
Finally, for $2q+1\leq d $ and $p \gtrsim 1$ we have that 
\begin{align*}
  (  \mathbb{E}|x_1|^q)^{1/q} \lesssim \sqrt{q}. 
\end{align*}
\end{lemma}
\begin{proof}
Since $x_j \thicksim \sqrt{(d-2)/d} t_d$, we have that 
$$
x_j= \frac{\sqrt{\frac{d-2}{d}} z_j}{\sqrt{\chi^2_{d,j}/d}}
$$
where $z$ denotes a standard Gaussian random variable and $\chi^2_{d,j}$ a chi-squared random variable with $d$ degrees of freedom that is independent of $z$. For $\beta \in \mathcal{S}^{p-1}$ we have, conditionally on the $\chi_{d,j}^2$-variables, that
\begin{align*}
    \langle X, \beta \rangle \bigg | (\chi_{d,j}^2)_{j=1}^p  \thicksim \mathcal{N} \left (0, \frac{d-2}{d}\sum_{i=1}^n \frac{\beta_i^2 d}{\chi_{d,j}^2} \right ).
\end{align*}
Hence, conditioning on the event where $ \chi_{d,j}^2  \geq d/2$ for all $j \in [p]$ and using independence of $z$ and the $\chi_{d,j}^2$-variables and that $\|\beta\|_2 \leq 1$ we obtain that 
\begin{align*}
    \mathbb{P} \left (  |\langle X,\beta \rangle|\leq \varepsilon \right )& \geq \mathbb{P} \left ( |z| \leq  \varepsilon \sqrt{d/(2(d-2))}\right ) \mathbb{P} \left ( \min_{j \in [p]} \chi_{d,j}^2 \geq d/2 \right )  \\
    & \gtrsim \varepsilon \mathbb{P} \left ( \min_{j \in [p]} \chi_{d,j}^2 \geq d/2 \right ) .
\end{align*}
It is left to lower bound the probability involving the minimum. By applying a lower tail bound for chi-square random variables, Lemma 1 in \cite{LaurentMassart00}, and a union bound we obtain  
\begin{align*}
\mathbb{P} \left ( \min_{j \in [p]} \chi_{d,j}^2 < d/2 \right ) \leq p \mathbb{P} \left (  \chi_{d,1}^2 < d/2 \right ) \leq pe^{-d/16} \leq p^{-1} \leq \frac{1}{2},
\end{align*} 
using the conditions on $d$ and $p$. \\

For the upper bound we argue similarly. We have
\begin{align*}
    \mathbb{P} \left ( |\langle X,\beta \rangle| \leq \varepsilon\right ) &  \leq \mathbb{P} \left ( \max_{j \in [p]} \chi^2_{d,j} \geq 2d \right ) + \mathbb{P} \left ( |Z| \leq \varepsilon \sqrt{2d/(d-2)}\right )  \\ & \lesssim \mathbb{P} \left ( \max_{j \in [p]} \chi^2_{d,j} \geq 2d \right ) + \varepsilon. 
\end{align*}
Applying an upper tail bound for chi-square random variables, Lemma 1 in \cite{LaurentMassart00}, we obtain
 $$\mathbb{P} \left ( \max_{j \in [p]} \chi^2_{d,j} > 2d \right ) \leq p \mathbb{P} \left (  \chi^2_{d,1} > 2d  \right ) \leq pe^{-d/16}  \leq p^{-1},$$ thus proving the claimed result. \\
 
Finally, for the claimed moment bound, integration by parts and for $\Gamma$ denoting the Gamma function, give
\begin{align*}
    \mathbb{E} |x_1|^q & = \frac{d^{q/2} \Gamma(\frac{q+1}{2}) \Gamma(\frac{d-q}{2})}{\pi^{1/2} \Gamma(\frac{d}{2}) }.
\end{align*}
We now only consider the case where $q$ is uneven, as the other case follows along the same lines. Indeed, since $d \gtrsim q$, applying Gautschi's inequality and using that $\Gamma(z)\leq z^{z-1/2}$ for $z\geq 1$ we have that
\begin{align*}
     \frac{d^{q/2} \Gamma(\frac{q+1}{2}) \Gamma(\frac{d-q}{2})}{ \Gamma(\frac{d}{2}) } & = \frac{d^{q/2}\Gamma(\frac{q+1}{2}) \Gamma(\frac{d-q}{2}) }{(\frac{d-2}{2} \cdot \dots \cdot \frac{d-q-1}{2} ) \Gamma ( \frac{d-q+1}{2})} \lesssim  \frac{ c^{(q-1)/2} d^{q/2}\Gamma(\frac{q+1}{2}) \Gamma(\frac{d-q}{2}) }{d^{(q-1)/2} \Gamma ( \frac{d-q+1}{2})} \\ 
     & \lesssim \frac{(cd)^{q/2}\Gamma(\frac{q+1}{2}) }{d^{q/2} } \lesssim  (cq)^{q/2}. 
\end{align*}
Taking the $q$-th root concludes the proof. 
\end{proof}

\begin{lemma} \label{Lemma margin studentt} 
Suppose that $X=(x_1, \dots, x_p)$ with $x_j \overset{i.i.d.}{\thicksim} \sqrt{\frac{d-2}{d}}t_{d}$ for $d \gtrsim \log(p)$ and $p \gtrsim 1$. Then, we have for any $\beta,\tilde \beta \in \mathcal{S}^{p-1}$
\begin{align*}
    \mathbb{P} \left ( \sgn{\langle X,\beta \rangle} \neq \sgn{\langle X, \tilde \beta \rangle}\right ) \gtrsim  \|\beta-\tilde \beta\|_2. 
\end{align*}
\end{lemma}
\begin{proof}
Since $x_j \thicksim \sqrt{(d-2)/d} t_d$, we have that 
$$
x_j= \frac{\sqrt{\frac{d-2}{d}} z_j}{\sqrt{\chi^2_{d,j}/d}}
$$
where $z$ denotes a standard Gaussian random variable and $\chi^2_{d,j}$ a chi-squared random variable with $d$ degrees of freedom that is independent of $z$. Denote $\beta_\chi=(\beta_j \sqrt{d/\chi^2_{d,j}})_{j\in [p]}$ and note that on the event $\{d/2 \leq \chi^2_{d,j} \leq 2d ~\forall 1 \leq j \leq p\}$ we have $\sqrt{1/2} \leq \|\beta_\chi\|_2 \leq \sqrt{2}$. 
Then, we have, conditioning on the $\chi^2_{d,j}$-variables and using Grothendieck's identity, Lemma 3.6.6. in \cite{vershynin2018high}, that
\begin{align*}
        \mathbb{P} \left ( \sgn{\langle X,\beta \rangle} \neq \sgn{\langle X, \tilde \beta \rangle}\right ) &  =  \frac{\mathbb{E} \arccos \left (\left \langle \frac{ \beta_\chi}{\|\beta_\chi\|_2}, \frac{\tilde \beta_\chi}{\|\tilde \beta_\chi\|_2} \right  \rangle \right )}{\pi} \\ & \geq \mathbb{E} \left \|\frac{ \beta_\chi}{\|\beta_\chi\|_2} -  \frac{\tilde \beta_\chi}{\|\tilde \beta_\chi\|_2}\right \|_2.
        \end{align*}
        We further bound, using that $\beta$ and $\tilde \beta$ have unit norm
        \begin{align*}
       \mathbb{E} \left \|\frac{ \beta_\chi}{\|\beta_\chi\|_2} -  \frac{\tilde \beta_\chi}{\|\tilde \beta_\chi\|_2}\right \|_2 & \gtrsim \mathbb{E} \left [ \left ( \|\beta_\chi \|\tilde \beta_\chi\|_2-\tilde \beta_\chi \|\beta_\chi\|_2 \|_2 \right ) \mathbf{1}(d/2 \leq \chi^2_{d,j} \leq 2d ~\forall 1 \leq j \leq p) \right ] \\
       & = \mathbb{E} \left [\sqrt{\sum_{j=1}^p  \frac{\left ( \beta_j\|\tilde \beta_\chi\|_2-\tilde \beta_j \|\beta_\chi\|_2 \right )^2}{\chi_{d,j}^2/d}} \mathbf{1}(d/2 \leq \chi^2_{d,j} \leq 2d ~\forall 1 \leq j \leq p) \right ] \\ 
       & \gtrsim \mathbb{E} \left [\left (\|\beta \|\tilde \beta_\chi\|_2-\tilde \beta \|\beta_\chi\|_2 \|_2 \right ) \mathbf{1}(d/2 \leq \chi^2_{d,j} \leq 2d ~\forall 1 \leq j \leq p)  \right ]\\
       & \geq \min_{a,b \in [1/2,2]}  \sqrt{a^2+b^2-2ab\langle \beta,\tilde \beta \rangle} \mathbb{P} \left ( d/2 \leq \chi^2_{d,j} \leq 2d ~\forall 1 \leq j \leq p\right ). 
\end{align*}
By Lemma 1 in \cite{LaurentMassart00}, a union bound and by our assumption on $d$ and $p$ we have that 
\begin{align*}
   \mathbb{P} \left (  d/2 \leq \chi^2_{d,j} \leq 2d ~\forall 1 \leq j \leq p \right ) \geq (1-2pe^{-d/16}) \geq 1/2.
\end{align*}
Hence, it is left to lower bound the quadratic equation $\min_{a,b \in [1/2,1]} (a^2+b^2-2ab\langle \beta,\tilde \beta \rangle)$. If $\langle \tilde \beta, \beta \rangle \leq 0$ it is clear that the minimum is attained at $a=b=1/2$. Conversely, if $\langle \tilde \beta,\beta \rangle >0$, we have since $0 <\langle \tilde \beta, \beta \rangle \leq 1$
\begin{align*}
    \min_{a,b \in [1/2,1]} (a^2+b^2-2ab\langle \beta,\tilde \beta \rangle) & \geq \min_{a \in \mathbb{R}, b \in [1/2,1]} (a^2+b^2-2ab \langle \beta, \tilde \beta \rangle ) \\ & = \min_{b \in [1/2,1]} b^2 \left(1-\langle \beta, \tilde \beta \rangle^2 \right)  \gtrsim (1-\langle \beta,\tilde \beta \rangle).
\end{align*}
Hence, summarizing, we have that
\begin{align*}
    \min_{a,b \in [1/2,2]} \sqrt{a^2+b^2-2ab\langle \beta,\tilde \beta \rangle} \gtrsim \sqrt{(2-2\langle \beta,\tilde \beta \rangle)} = \|\beta -\tilde \beta\|_2,
\end{align*}
thus concluding the proof. 
\end{proof}

\subsubsection{Proof of Corollary \ref{cor rate unimodal} }
\begin{proof} 
The proof of Corollary \ref{cor rate unimodal} follows mainly from Lemma \ref{lemma anti-concentration unimodal} below, which shows that the anti-concentration condition is satisfied for unimodal features with bounded density, and by noting that the weak moment assumption is satisfied for Laplace distributed features with $\zeta=1$, for student-t with at least $2 \log(p)+1$ degrees of freedom by Lemma \ref{lemma studentt conditions} with $\zeta=1/2$ and for uniform and Gaussian features with $\zeta=1/2$ as they are sub-Gaussian. 
\end{proof} 
\begin{lemma} \label{lemma anti-concentration unimodal}
Suppose that $X=(x_1, \dots,x_p)$ consists of i.i.d. symmetric and unit variance scalar random variables with density $f$. Suppose that $\|f\|_\infty \lesssim 1 $ and that $f$ is unimodal, i.e. $f(aw)\geq f(w)$ for any $0\leq a \leq 1$ and any $w \in \mathbb{R}$.  Then, we have for $0 \leq \varepsilon \leq 1$ that
\begin{align*}
   \sup_{\beta \in \mathcal{S}^{p-1}}  \mathbb{P} \left (  |\langle X, \beta \rangle| \leq \varepsilon \right ) \lesssim  \varepsilon^{1/2} \mathbb{E}|x_1|^3. 
\end{align*}
\end{lemma}
\begin{proof}
We consider two cases. If $\|\beta\|_\infty \leq \varepsilon^{1/2}$, then, by the Berry-Essen Theorem, e.g. Theorem 3.6. in \cite{ChenGoldSteinShao11book}
\begin{align}
    \mathbb{P} \left (  |\langle X, \beta \rangle| \leq \varepsilon \right ) & \lesssim \ \varepsilon + \sum_{j=1}^p |\beta_j|^3  \mathbb{E}|x_1|^3\lesssim  \varepsilon + \varepsilon^{1/2} \mathbb{E}|x_1|^3 \\ & \lesssim \varepsilon^{1/2} \mathbb{E}|x_1|^3. 
\end{align}
If $\|\beta\|_\infty \geq \varepsilon^{1/2}$, we argue as follows. Assume, without loss of generality, that $|\beta_1|\geq \varepsilon^{1/2}$. We note that since $x_1$ is unimodal that $\beta_1 x_1$ is unimodal, too. Then, since $\beta_1x_1$ is unimodal (see e.g. Theorem 1 in \cite{Anderson55}), we obtain that
\begin{align*}
    \mathbb{P} \left (  |\langle X, \beta \rangle| \leq \varepsilon \right ) & = \mathbb P\left (  |\beta_1 x_1+ \sum_{j=2}^p x_j \beta_j| \leq \varepsilon \right ) \leq \mathbb{P} \left ( |\beta_1 x_1|\leq \varepsilon \right ) \leq \mathbb{P} \left ( | x_1|\leq \varepsilon^{1/2} \right ) \\ & \lesssim \varepsilon^{1/2} \leq \varepsilon^{1/2} \mathbb{E}|x_1|^3, 
\end{align*}
as by Jensen's inequality $\mathbb{E}|x_1|^3 \geq (\mathbb{E}|x_1|^2)^{3/2}=1.$
\end{proof}

\section*{Acknowledgements}
GC and ML are funded in part by ETH Foundations of Data Science (ETH-FDS). 
Moreover, ML would like to thank C.S. Lorenz and M.D. Wong for helpful comments and FK would like to thank D. Fan for help with the Euler Cluster.

\bibliography{thesisBibliography}

\newcommand{\etalchar}[1]{$^{#1}$}
\begin{thebibliography}{WOBM17}

\bibitem[ABHZ16]{AwasthiBalcanHaghtalabZhang16}
P.~Awasthi, M.~Balcan, N.~Haghtalab, and H.~Zhang.
\newblock {Learning and 1-bit Compressed Sensing under Asymmetric Noise}.
\newblock In {\em Conference on Learning Theory (COLT)}, pages 152--192, 2016.

\bibitem[ALPV14]{AiLapanowskiPlanVershynin14}
A.~Ai, A.~Lapanowski, Y.~Plan, and R.~Vershynin.
\newblock {One-bit compressed sensing with non-Gaussian measurements}.
\newblock {\em Linear Algebra Appl.}, 441:222--239, 2014.

\bibitem[And55]{Anderson55}
T.W. Anderson.
\newblock {The Integral of a Symmetric Unimodal Function over a Symmetric
  Convex Set and Some Probability Inequalities}.
\newblock {\em Proc. Am. Math. Soc.}, 6(2):170--176, 1955.

\bibitem[BFLS98]{BartlettFreundLeeSchapire98}
P.~Bartlett, Y.~Freund, W.S. Lee, and R.E. Schapire.
\newblock Boosting the margin: a new explanation for the effectiveness of
  voting methods.
\newblock {\em Ann. Statist.}, 26(5):1651--1686, 1998.

\bibitem[BFN{\etalchar{+}}18]{BaraniukFoucartNeedell18}
R.~Baraniuk, S.~Foucart, D.~Needell, Y.~Plan, and M.~Wootters.
\newblock One-bit compressive sensing of dictionary-sparse signals.
\newblock {\em Inf. Inference}, 7:83--104, 2018.

\bibitem[BHMM19]{BelkinHsuMaMandal19}
M.~Belkin, D.~Hsu, S.~Ma, and S.~Mandal.
\newblock Reconciling modern machine-learning practice and the classical
  bias–variance trade-off.
\newblock {\em Proc. Natl. Acad. Sci. U.S.A.}, 116(32):15849--15854, 2019.

\bibitem[BLLT20]{BartlettLongLugosiTsigler20}
P.L. Bartlett, P.M. Long, G.~Lugosi, and A.~Tsigler.
\newblock Benign overfitting in linear regression.
\newblock {\em Proc. Natl. Acad. Sci. U.S.A.}, 117(48):30063--30070, 2020.

\bibitem[Bre98]{Breiman98}
L.~Breiman.
\newblock Arcing classifiers.
\newblock {\em Ann. Statist.}, 26(3):801--849, 1998.

\bibitem[Bre04]{Breiman04}
L.~Breiman.
\newblock Population theory for boosting ensembles.
\newblock {\em Ann. Statist}, 32(1):1--11, 2004.

\bibitem[B{\"u}h06]{Buehlmann06}
P.~B{\"u}hlmann.
\newblock Boosting for high-dimensional linear models.
\newblock {\em Ann. Statist.}, 34(2):559--583, 2006.

\bibitem[BV04]{boyd2004convex}
S.P. Boyd and L.~Vandenberghe.
\newblock {\em Convex Optimization}.
\newblock Cambridge University press, 2004.

\bibitem[BZ17]{BalcanZhang17}
M.~Balcan and H.~Zhang.
\newblock {Sample and Computationally Efficient Learning Algorithms under
  S-Concave Distributions}.
\newblock In {\em Conference on Neural Information Processing Systems (NIPS)},
  pages 4799--4808, 2017.

\bibitem[Car85]{Carl85}
B.~Carl.
\newblock {Inequalities of Bernstein-Jackson-type and the degree of compactness
  of operators in Banach spaces}.
\newblock {\em Ann. Inst. Fourier}, 35:79--118, 1985.

\bibitem[CDS98]{ChenDonohoSaunders98}
S.S. Chen, D.L. Donoho, and M.A. Saunders.
\newblock {Atomic Decomposition by Basis Pursuit}.
\newblock {\em SIAM J. Sci. Comput.}, 20:33--61, 1998.

\bibitem[CGLP13]{GuedonLecuePajor13}
D.~Chafa{\"{i}}, O.~Gu{\'e}don, G.~Lecu{\'e}, and A.~Pajor.
\newblock {\em Interactions between compressed sensing random matrices and high
  dimensional geometry}.
\newblock Soci\'et\'e Math\'ematique de France, 2013.

\bibitem[CGS11]{ChenGoldSteinShao11book}
L.H.Y. Chen, L.~Goldstein, and Q.M. Shao.
\newblock {\em {Normal Approximation by Stein's Method}}.
\newblock Springer, 2011.

\bibitem[CL20]{ChinotLerasle20}
G.~Chinot and M.~Lerasle.
\newblock Benign overfitting in the large deviation regime.
\newblock {\em arxiv preprint}, 2020.

\bibitem[CLvdG20]{ChinotLofflervandeGeer20}
G.~Chinot, M.~L\"offler, and S.~van~de Geer.
\newblock On the robustness of minimum norm interpolators and regularized
  empirical risk minimizers.
\newblock {\em arxiv preprint}, 2020.

\bibitem[DC96]{DruckerCortes96}
H.~Drucker and C.~Cores.
\newblock Boosting decision trees.
\newblock In {\em Advances in neural information processing systems (NIPS)},
  pages 479--485, 1996.

\bibitem[DKT20]{DengKammounThrampoulidis20}
Z.~Deng, A.~Kammoun, and C.~Thrampoulidis.
\newblock {A Model of Double Descent for High-dimensional Binary Linear
  Classification}.
\newblock {\em Inf. Inference, to appear}, 2020.

\bibitem[DM21]{DirksenMendelson21}
S.~Dirksen and S.~Mendelson.
\newblock {Non-Gaussian hyperplane tessellations and robust one-bit compressed
  sensing}.
\newblock {\em J. Eur. Math. Soc.}, 23(9):2913--2947, 2021.

\bibitem[DTKZ20]{DiakonikolasTzamosKontonisZarifis}
I.~Diakonikolas, C.~Tzamos, V.~Kontonis, and N.~Zarifis.
\newblock {Non-Convex SGD Learns Halfspaces with Adversarial Label Noise}.
\newblock In {\em Conference on Neural Information Processing Systems
  (NeurIPS)}, 2020.

\bibitem[FCG21]{FreiCaoGu21}
S.~Frei, Y.~Cao, and Q.~Gu.
\newblock {Agnostic Learning of Halfspaces with Gradient Descent via Soft
  Margins}.
\newblock In {\em International Conference on Machine Learning (ICML)}, 2021.

\bibitem[FHT00]{FriedmanHastieTibshirani00}
J.~Friedman, T.~Hastie, and R.~Tibshirani.
\newblock Additive logistic regression: a statistical view of boosting.
\newblock {\em Ann. Statist.}, 28(2):337--407, 2000.

\bibitem[FKMN21]{ForetNeyshaburKleinerMobahi21}
P.~Foret, A.~Kleiner, H.~Mobahi, and B.~Neyshabur.
\newblock Sharpness-aware minimization for efficiently improving
  generalization.
\newblock In {\em International Conference on Learning Representations (ICLR)},
  2021.

\bibitem[FS97]{FreundSchapire97}
Y.~Freund and R.E. Schapire.
\newblock {A Decision-Theoretic Generalization of On-Line Learning and an
  Application to Boosting}.
\newblock {\em J. Comput. Syst. Sci.}, 55(1):119--139, 1997.

\bibitem[GLSW06]{GordonLitvakSchuttWerner06}
Y.~Gordon, A.~Litvak, C.~Sch\"utt, and E.~Werner.
\newblock On the minimum of several random variables.
\newblock {\em Proc. Amer. Math. Soc.}, 134(12), 2006.

\bibitem[GN16]{GineNickl16}
E.~{Gin\'{e}} and R.~{Nickl}.
\newblock {\em Mathematical Foundations of Infinite-Dimensional Statistical
  Methods}.
\newblock Cambridge University Press, 2016.

\bibitem[HMRT19]{HastieMontanariRossetTibshirani19}
T.~Hastie, A.~Montanari, S.~Rosset, and R.J. Tibshirani.
\newblock Surprises in high-dimensional ridgeless least squares interpolation.
\newblock {\em arXiv preprint}, 2019.

\bibitem[Jia04]{Jiang04}
W.~Jiang.
\newblock {Process consistency for AdaBoost}.
\newblock {\em Ann. Statist.}, 32(1):13--29, 2004.

\bibitem[JLBB13]{JacquesLaskaBoufounosBaraniuk13}
L.~Jacques, J.N. Laska, P.T. Boufounos, and R.G. Baraniuk.
\newblock {Robust 1-Bit Compressive Sensing via Binary Stable Embeddings of
  Sparse Vectors}.
\newblock {\em IEEE Trans. Inform. Theory}, 59(4):2082--2102, 2013.

\bibitem[JT19]{JiTelgarsky19}
Z.~Ji and M.~Telgarsky.
\newblock The implicit bias of gradient descent on nonseparable data.
\newblock In {\em Conference on Learning Theory (COLT)}, pages 1772--1798,
  2019.

\bibitem[KKM20]{KrahmerKuemmerleMelnyk20}
F.~Krahmer, C.~K\"ummerle, and O.~Melnyik.
\newblock {On the Robustness of Noise-Blind Low-Rank Recovery from Rank-One
  Measurements}.
\newblock {\em arxiv preprint}, 2020.

\bibitem[KKR18]{KrahmerKuemmerleRauhut18}
F.~Krahmer, C.~K{\"u}mmerle, and H.~Rauhut.
\newblock {A Quotient Property for Matrices with Heavy-Tailed Entries and its
  Application to Noise-Blind Compressed Sensing}.
\newblock {\em arxiv preprint}, 2018.

\bibitem[KP02]{KoltchinskiiPanchenko02}
V.~Koltchinskii and D.~Panchenko.
\newblock {Empirical Margin Distributions and bounding the generalization error
  of combined classifiers}.
\newblock {\em Ann. Statist.}, 30(1):1--50, 2002.

\bibitem[KSW16]{KnudsonSaabWard16}
K.~Knudson, R.~Saab, and R.~Ward.
\newblock {One-Bit Compressive Sensing With Norm Estimation}.
\newblock {\em IEEE Trans. Inform. Theory}, 62(5):2748--2758, 2016.

\bibitem[LM00]{LaurentMassart00}
B.~Laurent and P.~Massart.
\newblock Adaptive estimation of a quadratic functional by model selection.
\newblock {\em Ann. Statist.}, 28(5):1302--1338, 2000.

\bibitem[LS20]{LiangSur20}
T.~Liang and P.~Sur.
\newblock {A Precise High-Dimensional Asymptotic Theory for Boosting and
  Minimum-$\ell_1$-Norm Interpolated Classifiers}.
\newblock {\em arxiv preprint}, 2020.

\bibitem[Men14]{mendelson2014learning}
S.~Mendelson.
\newblock Learning without concentration.
\newblock In {\em Conference on Learning Theory (COLT)}, pages 25--39, 2014.

\bibitem[MM21]{MeiMontanari19}
S.~Mei and A.~Montanari.
\newblock {The generalization error of random features regression: Precise
  asymptotics and double descent curve}.
\newblock {\em Comm. Pure Appl. Math., to appear}, 2021.

\bibitem[MNS{\etalchar{+}}21]{MuthukumarNarangSubramanianBelkinHsuSahai20}
V.~Muthukumar, A.~Narang, V.~Subramanian, M.~Belkin, D.~Hsu, and A.~Sahai.
\newblock {Classification vs regression in overparameterized regimes: Does the
  loss function matter?}
\newblock {\em J. Mach. Learn. Res.}, 22(222):1--69, 2021.

\bibitem[MPTJ07]{mendelson2005reconstruction}
S.~Mendelson, A.~Pajor, and N.~Tomczak-Jaegermann.
\newblock {Reconstruction and Subgaussian Operators in Asymptotic Geometric
  Analysis}.
\newblock {\em GAFA Geom. funct. anal.}, 17:1248--1282, 2007.

\bibitem[MRS13]{MukherjeeRudinSchapire13}
I.~Mukherjee, C.~Rudin, and R.E. Schapire.
\newblock {The Rate of Convergence of AdaBoost}.
\newblock {\em J. Mach. Learn. Res}, 14:2315--2347, 2013.

\bibitem[MRSY20]{MontanariRuanSohnYan20}
A.~Montanari, F.~Ruan, Y.~Sohn, and J.~Yan.
\newblock The generalization error of max-margin linear classifiers:
  High-dimensional asymptotics in the overparametrized regime.
\newblock {\em arxiv preprint}, 2020.

\bibitem[PV12]{PlanVershynin13CPAM}
Y.~Plan and R.~Vershynin.
\newblock One-bit compressed sensing by linear programming.
\newblock {\em Commun. Pure Appl. Math.}, 66(8):1275--1297, 2012.

\bibitem[PV13]{PlanVershynin13}
Y.~Plan and R.~Vershynin.
\newblock Robust 1-bit compressed sensing and sparse logistic regression: A
  convex programming approach.
\newblock {\em IEEE Trans. Inform. Theory}, 59(1):482--494, 2013.

\bibitem[Rio09]{Rio09}
E.~Rio.
\newblock {Moment Inequalities for Sums of Dependent Random Variables under
  Projective Conditions}.
\newblock {\em J. Theor. Probab.}, 22:146--163, 2009.

\bibitem[ROM01]{RaetschOnodaMueller01}
G.~R\"atsch, T.~Onoda, and K.R. M\"uller.
\newblock {Soft margins for AdaBoost}.
\newblock {\em Mach. Learn.}, 42:287--320, 2001.

\bibitem[RV08]{RudelsonVershynin08}
M.~{Rudelson} and R.~Vershynin.
\newblock {On sparse reconstruction from Fourier and Gaussian measurements.
  Communications on Pure and Applied Mathematics}.
\newblock {\em Comm. Pure Appl. Math.}, 61(8):1025--1045, 2008.

\bibitem[RZH04]{RossetZhuHastie04}
S.~Rosset, J.~Zhu, and T.~Hastie.
\newblock Boosting as a regularized path to a maximum margin classifier.
\newblock {\em J. Mach. Learn. Res.}, 5:941--973, 2004.

\bibitem[SHN{\etalchar{+}}18]{SoudryHofferNacsonGunasekarSrebro18}
D.~Soudry, E.~Hoffer, M.S. Nacson, S.~Gunasekar, and N.~Srebro.
\newblock The implicit bias of gradient descent on separable data.
\newblock {\em J. Mach. Learn. Res.}, 1:2822--2878, 2018.

\bibitem[SS99]{SchapireSinger99}
R.E. Schapire and Y.~Singer.
\newblock {Improved Boosting Algorithms Using Confidence-rated Predictions}.
\newblock {\em Mach. Learn.}, 37:297--336, 1999.

\bibitem[Tel13]{Telgarsky13}
M.~Telgarsky.
\newblock Margins, shrinkage, and boosting.
\newblock In {\em {International Conference on Machine Learning (ICML)}}, pages
  307--315, 2013.

\bibitem[Ver18]{vershynin2018high}
R.~Vershynin.
\newblock {\em {High-dimensional probability: An introduction with applications
  in data science}}, volume~47.
\newblock Cambridge university press, 2018.

\bibitem[WOBM17]{WynerOlsonBleichMease17}
A.J. Wyner, M.~Olson, J.~Bleich, and D.~Mease.
\newblock {Explaining the success of AdaBoost and random forests as
  interpolating classifiers}.
\newblock {\em J. Mach. Learn. Res.}, 18(48):1--33, 2017.

\bibitem[Woj10]{Wojtaszczyk10}
P.~Wojtaszczyk.
\newblock {Stability and Instance Optimality for Gaussian Measurements in
  Compressed Sensing.}
\newblock {\em Found. Comput. Math.}, 10:1--13, 2010.

\bibitem[WZZ{\etalchar{+}}13]{WanZeilerZhangLecunFergus13}
L.~Wan, M.~Zeiler, S.~Zhang, Y.~Lecun, and R.~Fergus.
\newblock {Regularization of Neural Networks using DropConnect}.
\newblock In {\em {International Conference on Machine Learning (ICML)}}, pages
  1058--1066, 2013.

\bibitem[ZBH{\etalchar{+}}17]{ZhangBengioHardtRechtVinyals17}
C.~Zhang, S.~Bengio, M.~Hardt, B~Recht, and O.~Vinyals.
\newblock Understanding deep learning requires rethinking generalization.
\newblock {\em International Conference on Learning Representations (ICLR)},
  2017.

\bibitem[Zha18]{Zhang18}
C.~Zhang.
\newblock Efficient active learning of sparse halfspaces.
\newblock In {\em Conference on Learning Theory (COLT)}, pages 1--26, 2018.

\bibitem[ZSA20]{ZhangShenAwasthi20}
C.~Zhang, J.~Shen, and P.~Awasthi.
\newblock Efficient active learning of sparse halfspaces with arbitrary bounded
  noise.
\newblock In {\em Advances in Neural Information Processing Systems 33 (NeurIPS
  2020)}, pages {7184--7197}, 2020.

\bibitem[ZY05]{ZhangYu05}
T.~Zhang and B.~Yu.
\newblock {Boosting with early stopping: Convergence and consistency}.
\newblock {\em Ann. Statist.}, 33(4):1538--1579, 2005.

\bibitem[ZYJ14]{ZhangYiJin14}
L.~Zhang, J.~Yi, and R.~Jin.
\newblock {Efficient Algorithms for Robust One-bit Compressive Sensing}.
\newblock In {\em International Conference on Machine Learning (ICML)}, pages
  820--828, 2014.

\end{thebibliography}
\bibliographystyle{alpha}

\appendix
\section{Dual formulation of the max $\ell_1$-margin} \label{dual formulation}

We use Lagrangian  duality to derive the dual version of the max $\ell_1$-margin. Recall that 
$$
\gamma  = \max_{\beta \neq 0} \min_{1\leq i \leq n} \frac{y_i \langle X_i, \beta \rangle}{\| \beta \|_1} = \frac{1}{\| \hat \beta \|} ,
$$
where we used Lemma~\ref{lemma margin sur}, recalling that
\begin{equation} \label{primal_problem}
    \hat \beta  \in \argmin_{\beta \in \mathbb R^p} \| \beta\|_1 \quad \text{subject to} \quad y_i \langle X_i,\beta \rangle \geq 1 .
\end{equation}
For every $\lambda \in \mathbb R^n $, define the Lagrangian $\mathcal L: \mathbb R^p \times \mathbb R^n \mapsto \mathbb R$ as
$$
\mathcal L(\beta,\lambda) = \| \beta \|_1 + \sum_{i=1}^n \lambda_i \big( 1- y_i \langle \beta,X_i \rangle    \big) .
$$
The dual problem of \eqref{primal_problem} is defined as
\begin{equation} \label{dual_problem}
    \sup_{\lambda \in \mathbb R^n_+} \inf_{\beta \in \mathbb  R^p} \mathcal L(\beta,\lambda).
\end{equation}
We have that 
\begin{align*}
   \inf_{\beta \in \mathbb R^p} \mathcal L(\beta,\lambda)&  = \inf_{\beta \in \mathbb R^p} \big \{  \| \beta \|_1 + \sum_{i=1}^n \lambda_i \big(1 - y_i\langle \beta,X_i \rangle   \big) \big \} \\
   & = \sum_{i=1}^n \lambda_i - \sup_{\beta \in \mathbb R^p} \big \{   \langle \beta, \sum_{i=1}^n \lambda_i y_i  X_i \rangle  - \| \beta\|_1   \big\}. 
\end{align*}
For any function $f: \mathbb R^p \mapsto \mathbb R$, the conjugate $f^*$ is defined as
\begin{equation} \label{def_conjugate}
    f^*(y) = \sup_{x \in \mathbb R^p} \big \{ \langle x,y \rangle - f(x)  \big \} .
\end{equation}
In particular (see \cite{boyd2004convex}, Example 3.26), when $f(\beta) = \|\beta\|_1$, we have that 
\begin{equation} \label{conjugate_norm}
f^*(y) = \left\{
    \begin{array}{ll}
        0 & \mbox{if } y \in B_{\infty}  \\
        \infty & \mbox{otherwise},
    \end{array}
\right.
\end{equation}
where $B_{\infty}$ is the unit ball with respect $\| \cdot \|_{\infty}$. 
From \eqref{def_conjugate} and~\eqref{conjugate_norm}, the dual problem \eqref{dual_problem} can be rewritten as
$$
\sup_{\lambda \in \mathbb R^n_+}  \sum_{i=1}^n \lambda_i \quad \textnormal{subject to} \quad \left\| \sum_{i=1}^n y_i \lambda_i X_i \right\|_{\infty} \leq 1.
$$
Since the $X_i$ are linearly independent with probability one and $p>n$, the Moore-Penrose inverse of $\mathbb X=[X_1,\dots, X_n]$ exists and hence there exists some  exists $\beta$ in $\mathbb R^p$ such that $y_i\langle X_i,\beta \rangle = 1$ for $i=1,\cdots,n$. Hence, Slater's condition is satisfied and consequently there is no duality gap. It follows that 
$$
\gamma = \frac{1}{\| \hat \beta\|_1 } = \inf_{w:~w_i \geq 0 ~\forall i \in [n],\|w\|_1=1} \left\| \sum_{i=1}^n w_i y_i X_i \right\|_{\infty}.
$$

\section{Extra Lemmas}

\subsection{Lemma~\ref{control infty norm log moments}}

\begin{lemma} \label{control infty norm log moments}
Let $X = (x_{1},\cdots, x_{p})^T$ be a random vector where the $x_{j}$'s are i.i.d random  variables that satisfy the weak moment assumption with $\zeta \geq 1/2$. 
Then, with probability at least $1-p^{-2}$ we have that \begin{align}\| X \|_{\infty} \lesssim \log^{\zeta}(p).
\end{align}
Moreover, let $X_1, \dots, X_n$, $n \leq p$, be $n$ i.i.d. copies of $X$ and $\beta \in \mathcal{S}^{p-1}$. Then, we have additionally with probability at least $1-n^{-2}$
\begin{align*}
 \max_{i \in [n]} |\langle X_i, \beta^* \rangle| \lesssim \log^{1/2+ \zeta}(n).
\end{align*}
\end{lemma}

\begin{proof}

We have that \begin{align*}
   \left (  \mathbb{E} (\max_{j \in [p]} |x_{j}|)^{q} \right )^{1/q} \leq p^{1/q} \left ( \mathbb{E} |x_{1}|^{q} \right )^{1/q}  \lesssim p^{1/q} {q}^{\zeta }. 
\end{align*}
Hence, by Markov's inequality,
\begin{align*}
    \mathbb{P} \left ( \|X\|_\infty > t \right ) \leq e^{\log(p)+cq + \zeta q\log(q)-q \log(t)}.
\end{align*}
Choosing  $q= \log(p)$ and $t \asymp \log^{\zeta}(p)$ concludes the proof of the first claim.

For the second claim we argue as follows. By Rio's version of the Marcinkiewicz-Zygmund inequality, Theorem 2.1. in \cite{Rio09}, we have that
\begin{align*}
 \left (  \mathbb{E} 
\max_{i \in [n]} |\langle X_i, \beta \rangle|^{q} \right )^{1/q} &  \leq n^{1/q} \left (\mathbb{E} |\langle X, \beta \rangle|^{q} \right )^{1/q} \leq n^{1/q} {q}^{1/2} \left (\sum_{j=1}^p |\beta_j|^2 (\mathbb{E} |x_{j}|^{q})^{2/q}  \right )^{1/2}\\ & \lesssim n^{1/q}  {q}^{1/2+\zeta}.
\end{align*}

Arguing as before with $q=\log(n) $ concludes the proof. 
\end{proof}

\subsection{Proposition~\ref{quotient property}}

\begin{proposition}[Theorem 5~\cite{KrahmerKuemmerleRauhut18}] \label{quotient property}
Let $X_1,\cdots, X_n$ be i.i.d random vectors distributed as $X = (x_{1},\cdots, x_{p})^T$, where the $x_{j}$'s are i.i.d symmetric, zero mean and unit variance random variables that satisfy the weak moment assumption.
For  $ Z \in \mathbb R^n$ and $\mathbb{X} = [X_1, \cdots X_n]$ define
$$
\hat \nu := \argmin_{\beta \in \mathbb R^p} \| \beta \|_1 \quad \text{such that} \quad \mathbb X^T \beta =  Z.
$$
Assume that $p \gtrsim n$. Then,  with probability at least $1-2\exp(-2n)$ we have that
$$
\| \hat \nu \|_1 \lesssim \frac{\|Z \|_2}{\sqrt{\log(ep/n) }} +  \|  Z \|_{\infty} .
$$
\end{proposition}

\subsection{Rademacher complexity under weak moment assumption} 

\begin{proposition} \label{control Rademacher complexity low moments}
Assume that $\mathbb{X}=(X_i)_{i \in [n]}$ has i.i.d. zero mean and unit variance entries and satisfies the weak moment assumption with $\zeta \geq 1/2$. For $a \in \mathbb{N}$ we have that 
$$
\mathbb E \sup_{\beta \in  aB_1^p \cap B_2^p} \frac{1}{n} \sum_{i=1}^n \sigma_i \langle X_i , \beta \rangle \lesssim   a \sqrt \frac{ \log(p)}{n}. 
$$
\end{proposition}
The proof of Proposition \ref{control Rademacher complexity low moments} uses the following bound for sums of order statistics and will be presented below. 
\begin{lemma} \label{lemma ordering}
Assume that $X = (x_{1},\cdots, x_{p})^T$ has i.i.d symmetric, zero mean and unit variance entries  that satisfy the weak moment assumption with $\zeta \geq 1/2$.
Then, for all $1\leq k \leq p$ we have
$$
\mathbb E \left( \sum_{i=1}^k (x_i^*)^2 \right)^{1/2} \lesssim  \log^{\zeta}(p) \sqrt k  ,
$$
where $(x_i^*)_i^p$ is a monotone non-increasing rearrangement of $(| x_i |)_{i=1}^p$. 
\end{lemma}

\begin{proof}
The proof is a small adaptation from Lemma 6.5. in~\cite{mendelson2014learning} where $\zeta=1/2$ is assumed.  Fix $1 \leq j \leq p$, $1 \leq q \leq \log(p)$ and $t>0$. We have by the weak moment assumption for some $c_1>0$
$$
\mathbb P \left( x_j^* \geq t \right) \leq {p \choose j} \mathbb P^j \left( |x_1 | \geq t  \right) \leq {p \choose j}  \left( \frac{\mathbb E | x_1 |^q}{t^q} \right)^j \leq {p \choose j}  \left( \frac{c q^{q\zeta}}{t^q} \right)^j
$$
Since ${p \choose j} \leq \exp(j \log(p))$, taking $q = \log(p)$ we get
$$
\mathbb P \left( x_j^* \geq t \right) \leq \left( \frac{c  \log^{\zeta}(p)}{t} \right)^{j\log(p)}
$$
Hence, integrating out the tails and using Jensen's inequality
it follows that

$$
\mathbb E \left( \sum_{i=1}^k \left( (x_i^*)^2 \right) \right)^{1/2} \leq \left( \mathbb E\sum_{i=1}^k \left( x_i^* \right)^2 \right)^{1/2} \lesssim  \log^{\zeta}(p) \sqrt k .  
$$
\end{proof}
\noindent \textbf{Proof of Proposition \ref{control Rademacher complexity low moments}} \\
\begin{proof}
From Equation 3.1 in~\cite{mendelson2005reconstruction}, we have that 
\begin{align*}
  \mathbb E \sup_{\beta \in aB_1^p \cap B_2^p} \frac{1}{n} \sum_{i=1}^n \sigma_i \langle X_i , \beta \rangle & \leq 2 \mathbb E \sup_{\beta \in B_0^p(a^2) \cap B_2^p} \frac{1}{n} \sum_{i=1}^n \sigma_i \langle X_i , \beta \rangle \\
    & = \frac{2}{\sqrt n} \mathbb E \sup_{\beta \in B_0^p(a^2) \cap B_2^p} \langle W , \beta \rangle
\end{align*}
where $B_0^p(a^2) = \{ \beta \in \mathbb{R}^p : \|\beta \|_0 \leq a^2 \}$ and $W =n^{-1/2}\sum_{i=1}^n \sigma_i X_i$ and it follows that
\begin{align} \label{ref eq rademacher order}
  \mathbb E \sup_{\beta \in aB_1^p \cap B_2^p} \frac{1}{n} \sum_{i=1}^n \sigma_i \langle X_i , \beta \rangle \leq \frac{2}{\sqrt n} \mathbb E \left(\sum_{i=1}^{a^2} (W_i^*)^2 \right)^{1/2} .
\end{align}
The $(W_i)$'s are centered random variables. For $1 \leq q \leq  \log(p)  $, using the Khintchine-Kahane inequality, Proposition 3.2.8 in~\cite{GineNickl16}, and Jensen's inequality
\begin{align*}
    \left(\mathbb E | W_j |^q \right)^{1/q}  & = \left( \mathbb{E} \left | \frac{1}{\sqrt n }\sum_{i=1}^n \sigma_i X_{i,j}  \right|^q \right)^{1/q} \lesssim \sqrt{q} \mathbb{E} \left( \frac{1}{n}\sum_{i=1}^n X_{i,j}^2 \right)^{1/2} \\
    &\leq \sqrt{q} \left( \mathbb{E} X_{1,1}^2 \right)^{1/2} \leq  \sqrt{q}
\end{align*}
Thus, applying Lemma~\ref{lemma ordering} with $\zeta=1/2$ to bound \eqref{ref eq rademacher order} we have that 
$$
  \mathbb E \sup_{\beta \in aB_1^p \cap B_2^p} \frac{1}{n} \sum_{i=1}^n \sigma_i \langle X_i , \beta \rangle\lesssim a \sqrt \frac{ \log(p)}{n} ,
$$
concluding the proof. 
\end{proof}
\end{document}